\documentclass[a4paper, reqno]{amsart}

\usepackage[all]{xy}
\usepackage[english]{babel}
\usepackage[utf8]{inputenc}
\usepackage{amssymb,amsmath,amsthm}
\usepackage{amsfonts}
\usepackage{graphicx}
\usepackage{psfrag}
\usepackage{a4wide}
\usepackage[dvipsnames]{xcolor}
\usepackage[hidelinks]{hyperref}
\usepackage{csquotes}
\usepackage[alphabetic]{amsrefs}
\usepackage{wrapfig}

\allowdisplaybreaks

\newcommand{\id}{\mathrm{id}}
\newcommand{\Ric}{\mathrm{Ric}}

\newcommand{\e}{\epsilon}

\renewcommand{\S}{\Sigma}

\newcommand{\n}{\nabla}

\renewcommand{\L}{\mathcal{L}}
\newcommand{\E}{\mathcal E}
\newcommand{\g}{{\bar g}}
\renewcommand{\o}{\omega}
\newcommand{\supp}{\mathrm{supp}}
\newcommand{\tr}{\mathrm{tr}}
\renewcommand{\a}{\alpha}

\renewcommand{\b}{\beta}
\renewcommand{\d}{\partial}
\newcommand{\abs}[1]{\left\lvert#1\right\rvert}
\newcommand{\norm}[1]{\left\lVert#1\right\rVert}

\renewcommand{\div}{\mathrm{div}}
\newcommand{\grad}{\mathrm{grad}}
\newcommand{\Hess}{\mathrm{Hess}}
\renewcommand{\H}{\mathcal H}
\renewcommand{\E}{\mathcal E}
\renewcommand{\Re}{\mathrm{Re}}

\theoremstyle{plain}
\newtheorem{thm}{Theorem}[section]
\newtheorem{prop}[thm]{Proposition}
\newtheorem{lemma}[thm]{Lemma}
\newtheorem{cor}[thm]{Corollary}
\theoremstyle{definition}
\newtheorem{definition}[thm]{Definition}

\newtheorem{remark}[thm]{Remark}
\newtheorem{example}[thm]{Example}
\newtheorem{counterexample}[thm]{Counter Example}

\newtheorem{assumption}[thm]{Assumption}
\newtheorem{conjecture}[thm]{Conjecture}

\newcommand{\R}[0]{\mathbb{R}}							% Real numbers
\newcommand{\N}[0]{\mathbb{N}}							% Natural numbers
\newcommand{\der}[2]{\frac{d}{d #1}\Big|_{#1 = #2}}

\author{Oliver Lindblad Petersen}
\title{Wave equations with initial data on compact Cauchy horizons}
\address{University of Hamburg, Department of Mathematics, Bundesstraße 55, 20146 Hamburg, Germany}
\email{oliver.petersen@uni-hamburg.de}

\begin{document}
\hbadness=100000
\vbadness=100000

\begin{abstract}
We study the following problem: Given initial data on a compact Cauchy horizon, does there exist a unique solution to wave equations on the globally hyperbolic region? 
Our main results apply to any spacetime satisfying the null energy condition and containing a compact Cauchy horizon with surface gravity that can be normalised to a non-zero constant.
Examples include the Misner spacetime and the Taub-NUT spacetime.
We prove an energy estimate close to the Cauchy horizon for wave equations acting on sections of vector bundles.
Using this estimate we prove that if a linear wave equation can be solved up to any order at the Cauchy horizon, then there exists a unique solution on the globally hyperbolic region.
As a consequence, we prove several existence and uniqueness results for linear and non-linear wave equations without assuming analyticity or symmetry of the spacetime and without assuming that the generators close.
We overcome in particular the essential remaining difficulty in proving that vacuum spacetimes with a compact Cauchy horizon with constant non-zero  surface gravity necessarily admits a Killing vector field.
This work is therefore related to the strong cosmic censorship conjecture.
\end{abstract}

\maketitle

\tableofcontents
\begin{sloppypar}

\section{Introduction}

The purpose of this paper is to present new methods to treat the characteristic Cauchy problem for wave equations with initial data on a \emph{smooth}, \emph{compact}, \emph{totally geodesic} Cauchy horizon \emph{with surface gravity that can be normalised to a non-zero constant}.
To the best of our knowledge, all known examples of compact Cauchy horizons in (electro-)vacuum spacetimes satisfy these conditions.
In fact, it was recently proven that any compact Cauchy horizon in a spacetime satisfiying the null energy condition is necessarily smooth and totally geodesic (see \cite{Larsson2014} and  \cite{Minguzzi2015}).

Since Cauchy horizons are lightlike hypersurfaces, the metric degenerates at the Cauchy horizon.
This paper therefore concerns a geometric singular initial value problem for wave equations.
We prove:
\begin{enumerate}
	\item If a linear wave equation can be solved up to \emph{any order} at the Cauchy horizon, it can be solved uniquely on the globally hyperbolic region (Theorem \ref{mainthm1}).
	In particular, if a solution to a linear wave equation vanishes up to any order at the Cauchy horizon, then it vanishes on the globally hyperbolic region.
	\item If a solution to what we call an \emph{admissible} linear wave equation vanishes at the Cauchy horizon, then it vanishes on the globally hyperbolic region (Theorem \ref{mainthm2}).
	\item Given any initial data to an admissible linear wave equation for scalar valued functions, there exists a unique solution on the globally hyperbolic region (Theorem \ref{mainthm3}).
	\item Given any initial data to an admissible semi-linear wave equation for scalar valued functions, there exists a unique solution on a neighbourhood of the Cauchy horizon (Theorem \ref{mainthm4}).
\end{enumerate}
The statements $(3)$ and $(4)$ require a condition on the Ricci curvature at the Cauchy horizon, implied for example by the dominant energy condition. 
All four statements hold in any spacetime dimension greater or equal to $2$. 
Rather surprisingly, simple counter examples show that the statements $(2-4)$ were false for general wave operators. 
We also give counter examples implying that all four statements were false if we allowed for Cauchy horizons with vanishing surface gravity.

Our results have a natural application in general relativity. 
The strong cosmic censorship conjecture says that the maximal globally hyperbolic vacuum (or suitable matter) developments of generic relativistic initial data is inextendible. 
In other words, the maximally globally hyperbolic hyperbolic vacuum spacetimes which are extendible over a Cauchy horizon are expected to be \emph{non-generic}. 
Moncrief and Isenberg conjectured in \cite{MoncriefIsenberg1983} that a vacuum spacetime with a \emph{compact} Cauchy horizon necessarily admits a Killing vector field in the globally hyperbolic region.
Together with István Rácz we show in \cite{PetersenRacz2018} that the Killing equation can be solved up to any order at a compact Cauchy horizon, if the spacetime is vacuum and the surface gravity can be normalised to a non-zero constant.
Applying the main result of this paper, Theorem \ref{mainthm1}, then proves the Moncrief-Isenberg conjecture under the assumption that the surface gravity of the Cauchy horizon can be normalised to a non-zero constant.
This generalises classical results by Moncrief-Isenberg and Friedrich-Rácz-Wald, who assume that the generators are closed or densely fill a 2-torus (\cite{MoncriefIsenberg1983} \cite{FRW1999} \cite{MoncriefIsenberg2018}).
In particular, \emph{generic} maximally globally hyperbolic vacuum developments of asymptotically flat or compact initial data cannot be extended over a compact Cauchy horizon with surface gravity that can be normalised to a non-zero constant.
This is a natural step towards the strong cosmic censorship conjecture without symmetry assumptions.
We postpone a more detailed discussion of this until we have presented the precise formulation of our main results.

Let $(M, g)$ denote a spacetime, i.e.\ a connected time-oriented Lorentzian manifold. 
Assume that $\S$ is an acausal topological hypersurface of $M$, such that $\S \subset M$ is a closed subset.
Let $n+1$ denote the dimension of $M$, we only require that $n+1 \geq 2$.
The domain of dependence $D(\S) \subset M$ (also called Cauchy development) is a globally hyperbolic submanifold of $M$ and $\S$ is a Cauchy hypersurface for $D(\S)$ (see \cite{O'Neill1983}*{Lem. 43}). 
The classical well-posedness statement for the Cauchy problem says that if $\S$ is smooth and spacelike, then linear wave equations can be solved uniquely on $D(\S)$ if one specifies initial data on $\S$, see e.g.\ \cite{BaerGinouxPfaeffle2007}*{Thm. 3.2.11}. 
In this paper, we instead specify the initial data on a part of the boundary of $D(\S)$. 
By \cite{O'Neill1983}*{Prop. 14.53}, the boundary of $D(\S)$ can be divided into disjoint sets
\[
	\d D(\S) = \H_+ \sqcup \H_-,
\] 
where $\H_\pm := \overline{D_\pm(\S)} \backslash D_\pm(\S)$ are called the \emph{future} and the \emph{past Cauchy horizon of $\S$} respectively. 
$\H_+$ and $\H_-$ are closed achronal lightlike Lipschitz hypersurfaces (if non-empty) of $M$. 
Let from now on $\H$ denote the future or the past Cauchy horizon of $\S$.
If $\H$ is a smooth hypersurface, then $\H \sqcup D(\S)$ is a smooth manifold with boundary. 
In our case, $\H$ will be smooth and compact.
Therefore there will be closed or almost closed lightlike curves in $\H$, which implies that $M$ is not globally hyperbolic.
Our results therefore concern wave equations on \emph{non-globally hyperbolic spacetimes}.
A first example is given by the classical Misner spacetime, where $M = \R \times S^1$ with coordinates $t$ and $x$ and $g = 2dtdx + tdx^2$. 
As shown in Figure \ref{fig: Misner}, the Misner spacetime admits a compact past Cauchy horizon $\H := \{t = 0\}$ of any level set $\S := \{t = C\}$, where $C$ is a positive constant (see Example \ref{ex: Misner spacetime} for more details).

Let us explain what we mean by \enquote{normalising the surface gravity} of $\H$.
Any lightlike vector field $V$ tangent to $\H$ is pre-geodesic, i.e. there is a smooth function $\kappa$ such that
\[
	\n_V V = \kappa V
\]
on $\H$.
Motivated by the corresponding equation on black hole horizons, we call $\kappa$ the \enquote{surface gravity} associated with $V$. 
It is important, however, to note that there is no canonical way to normalise the surface gravity on a compact Cauchy horizon.
Any vector field $fV$ for a function $f$ will give a different surface gravity.

\begin{assumption} \label{as: Cauchy_horizon}
Assume from now on that $\H$ is a (non-empty) smooth, compact, totally geodesic future or past Cauchy horizon of $\S$ and that there is a nowhere vanishing lightlike vector field $V$ tangent to $\H$ such that
\[
	\n_V V = \kappa V
\]
for a \textbf{non-zero constant} $\kappa$.
\end{assumption}

\noindent 
By substituting $V$ with $\frac1 \kappa V$, we may from now on assume that $\kappa = 1$.
The integral curves of $V$ (or their reparametrisations as geodesics) are called \emph{generators} of the Cauchy horizon.
% Note by Remark \ref{rmk: vanishing_surface_gravity} that if there is a nowhere vanishing lightlike vector field $W$ tangent to $\H$ such that $\n_W W = 0$, then the surface gravity cannot be normalised to a non-zero constant.

\begin{figure} \label{fig: Misner}
  \begin{center}
    \includegraphics[scale = 0.7]{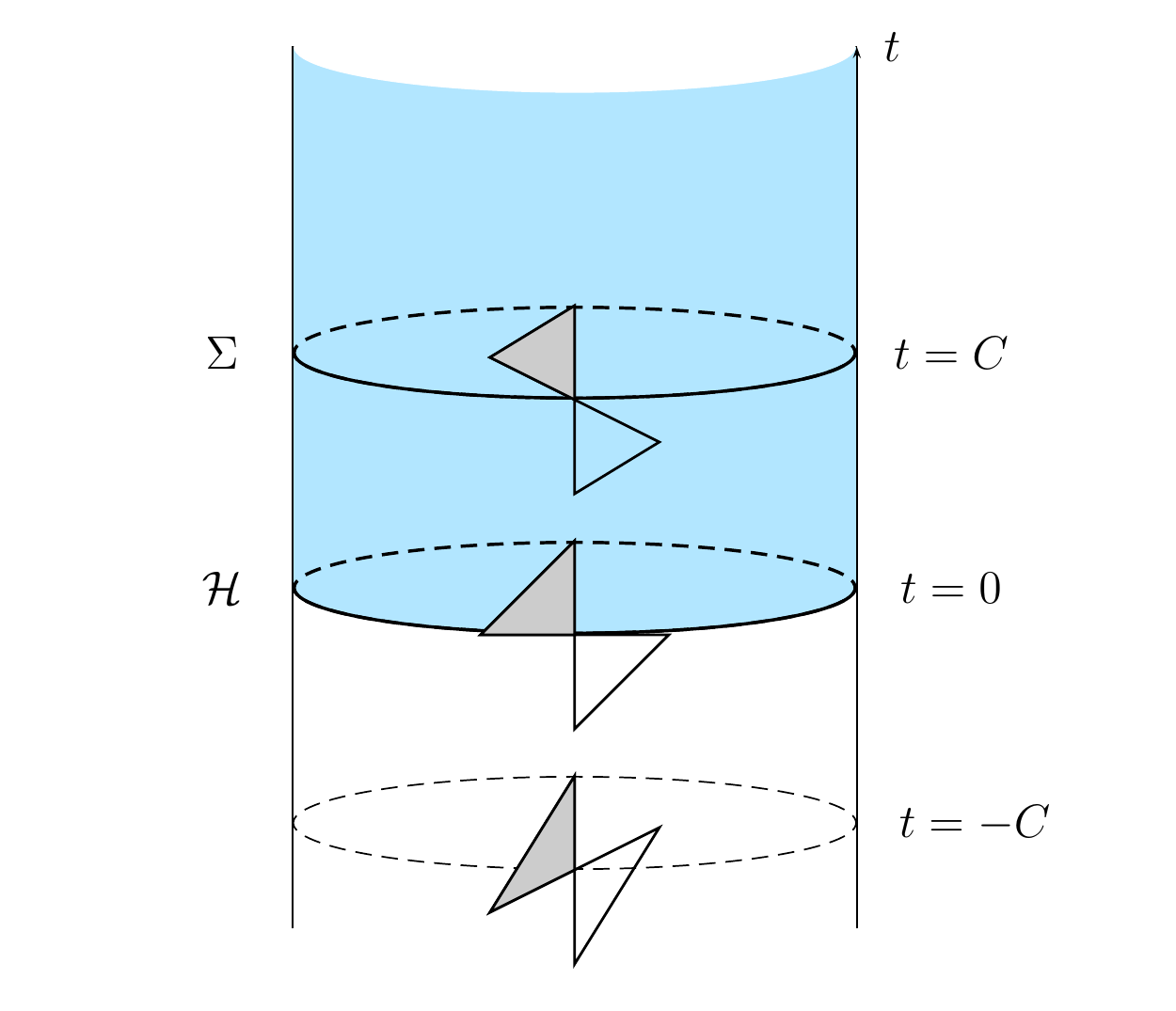}
  \end{center}
  \caption{The Misner spacetime (Example \ref{ex: Misner spacetime}) with three hypersurfaces of different causality type. $\S$ is a closed acausal (spacelike) hypersurface and $\H$ is the corresponding (lightlike) past Cauchy horizon. The blue region is the domain of dependence of $\S$.}
\end{figure}

\begin{remark} \label{rmk: null_energy_condition}
Let us emphasise that Assumption \ref{as: Cauchy_horizon} is automatically satisfied if $\H$ is non-empty and compact, the surface gravity can be normalised to a non-zero constant and $M$ satisfies the null energy condition, i.e.\ $\Ric(L, L) \geq 0$ for all lightlike vectors $L \in TM$.
This follows by an important recent result proven independently by Larsson in  \cite{Larsson2014}*{Cor. 1.43} and Minguzzi in  \cite{Minguzzi2015}*{Thm. 18}.
\end{remark}

\begin{remark}
The Cauchy horizons in the Misner spacetime and in the Taub-NUT spacetime satisfy Assumption \ref{as: Cauchy_horizon}. 
Therefore our results apply to these two important classes of spacetimes.
We will discuss these spacetimes and one further important example with non-closed generators in Section \ref{sec: Remarks}.
\end{remark}

For a subset $N \subset M$ and a vector bundle $F \to M$, let 
\[
	C^\infty(N, F)
\]
denote the space of smooth sections in $F$ defined on $N$.
We write $C^\infty(N)$ instead of $C^\infty(N, F)$ whenever it is clear what vector bundle is meant.

\begin{definition}[Wave operator]
A second order linear differential operator $P$ acting on sections in a vector bundle $F$ is called a \emph{wave operator} if the principal symbol is given by 
\[
	\sigma(P, \xi) = -g(\xi, \xi)
\]
for all $\xi \in T^*M$.
Equivalently, in local coordinates, a wave operator is given by
\[
	P = -g^{\a \b}\d_\a \d_\b + \text{lower order terms}.
\]
\end{definition} 

\begin{assumption}
Let from now on $F \to M$ denote a real or complex vector bundle $F \to M$ with connection $\n$ and let $P$ be a wave operator acting on sections in $F$.
Together with the Levi-Civita connection, we get an induced connection $\n$ on $(T^*M)^j \otimes F$ for each $j \in \N$. 
\end{assumption}

\textbf{Global existence and uniqueness given asymptotic solution.}
Our first main result says that if a linear wave equation can be solved asymptotically up to any order at $\H$, then there is a unique solution on $\H \sqcup D(\S)$ with corresponding asymptotic expansion.

\begin{thm}[Global existence and uniqueness given asymptotic solution] \label{mainthm1}
Let $f \in {C^\infty(\H \sqcup D(\S), F)}$. 
Assume that there are sections $(w^N)_{N \in \N} \subset C^\infty(\H \sqcup D(\S), F)$ such that 
\begin{align*}
	\n^k(Pw^N - f)|_\H &= 0,
\end{align*}
and $\n^k (w^N -w^{N+1})|_\H = 0$ for all $N \in \N$ and $k \leq N$. 
Then there exists a unique $u \in {C^\infty(\H \sqcup D(\S), F)}$ such that
\begin{align*}
	Pu &= f \text{ on } D(\S), \\
	\n^ku|_\H &= \n^kw^{N}|_\H,
\end{align*}
for all $N \in \N$ and $k \leq N$.
\end{thm}

\noindent
The notation $\n^ka|_\H = \n^k b|_\H$ means that 
\[
	\n^k_{X_1, \dots, X_k}a|_\H = \n^k_{X_1, \dots, X_k}b|_\H
\]
for all $X_1, \dots, X_k \in TM|_\H$.
We think of $w^N$ as the asymptotic expansion of the actual solution $u$ up to order $N$.
Let us show in a simple example how Theorem \ref{mainthm1} can be applied.
\begin{example} \label{ex: ex_main1}
Let $(M = \R \times S^1, g = 2dtdx + tdx^2)$ be the already mentioned Misner spacetime (see Figure \ref{fig: Misner} and Example \ref{ex: Misner spacetime}) with past Cauchy horizon $\H = \{t = 0\}$.
We show later in Example \ref{ex: Misner spacetime} that $\H$ actually satisfies our assumptions.
Consider the initial value problem
\begin{align*}
	Pu &= f, \\
	u|_{\H} &= u_0,
\end{align*} 
where $u_0$ is a given smooth function on $S^1$.
To show existence of a solution, we may choose 
\[
	w^N(t, x) := \sum_{j = 0}^{N+1}\frac{u_j(x)t^j}{j!}.
\]
If for example $P = \Box + 1$, then $P = t \d_t\d_t - 2\d_x\d_t + \d_t +1$ and the equations $\n^k(Pw^N - f)|_{\H} = 0$ are equivalent to
\[
	(- 2 \d_x + (k+1))u_{k+1} + u_k - (\d_t)^kf|_{t = 0} = 0
\]
on $S^1$ for each $k \in \N$.
These equations can now be solved iteratively on $S^1$ in a unique way for each order $k \in \N$.
We have thus computed the asymptotic solutions $w^N$ and Theorem \ref{mainthm1} guarantees a solution on the region $\H \sqcup D(\S) = [0, \infty) \times S^1$ to the initial data $u_0$.
Note that this solution in fact unique, since any asymptotic expansion must be of the above form.
In Theorem \ref{mainthm2}, Theorem \ref{mainthm3} and Theorem \ref{mainthm4} we generalise this procedure in different directions.
\end{example}

\noindent The proofs of the other main results of this paper will rely heavily on Theorem \ref{mainthm1} using techniques which generalise Example \ref{ex: ex_main1}.
Note also that we in Example \ref{ex: ex_main1} only specify the restriction of the solution to the hypersurface as initial data and not the transversal derivative as one would do on a spacelike hypersurface.

Choosing $f = 0 = w^N$ in Theorem \ref{mainthm1}, we get the following unique continuation statement.

\begin{cor}[Unique continuation] \label{cor: unique_cont}
Assume that $u \in C^\infty(\H \sqcup D(\S), F)$ satisfies
\begin{align*}
	Pu &= 0 \text{ on } D(\S),\\
	\n^k u|_\H &= 0,
\end{align*} 
for all $k \in \N$.
Then $u = 0$ on $\H \sqcup D(\S)$.
\end{cor}

\noindent 
Statements reminiscent of these were up to now only known in the analytic setting (or in special situations), using a Cauchy-Kowalevski argument.
Our argument relies instead on an energy estimate close to the Cauchy horizon, which allows us to drop the highly restrictive analyticity assumption. 
It is also interesting to note that Theorem \ref{mainthm1} were false if we allowed for vanishing surface gravity, see Counter Example \ref{counterex: non-zero}. 

\textbf{Uniqueness for admissible wave equations.}
In Corollary \ref{cor: unique_cont} it was assumed that the solution vanishes up to infinite order at the Cauchy horizon.
It is however well-known that the solution to a linear wave equation on globally hyperbolic spacetimes is uniquely determined by its restriction to certain lightlike hypersurfaces, see \cite{BaerWafo2014}*{Thm. 23}.
Since the spacetimes we consider here are not globally hyperbolic, \cite{BaerWafo2014}*{Thm. 23} does not apply.
It turns out that under a certain condition on the wave operator the solution is uniquely determined by its restriction on $\H$, whereas for general wave operators the solution is not unique.
Let us define this class of \emph{admissible} wave operators. 
Note that $P$ is a wave operator if and only if
\[
	Pu = \n^*\n u + B(\n u) + A(u),
\]
for some smooth homomorphism fields $B$ from $T^*M \otimes F$ to $F$ and $A$ from $F$ to $F$ and where
\begin{equation}
	\n^*\n := -\tr_g(\n^2). \label{eq: connection-dAlembert}
\end{equation}

\begin{definition}[Admissible wave operator] \label{def: admissible}
Assume that $a$ is a symmetric or hermitian positive definite metric on $F$ such that
\begin{equation}
	\n_Va|_\H(w, w) + a|_\H(B(g(V, \cdot) \otimes w), w) \leq 0, \label{eq: admissible}
\end{equation}
for all $w \in F|_\H$. 
Then we call $P$ an \emph{admissible wave operator with respect to $\H$}.
\end{definition}

\begin{thm}[Uniqueness for admissible wave equations] \label{mainthm2}
Let $P$  be an admissible wave operator with respect to $\H$ in the sense of Definition \ref{def: admissible}. Assume that $u \in C^\infty(\H \sqcup D(\S), F)$ such that
\begin{align*}
	Pu &= 0 \text{ on } D(\S),\\
	u|_\H &= 0.
\end{align*} 
Then $u = 0$ on $\H \sqcup D(\S)$.
\end{thm}

\begin{remark}
Remarkably, the conclusion of Theorem \ref{mainthm2} fails for general wave operators.
There is one particularly important example. 
Let $F := TM$ and let $\n$ be the Levi-Civita connection.
Since $g$ is not positive definite on $F$, the operator $\n^*\n$ is \emph{not} a priori admissible in the sense of Definition \ref{def: admissible}. 
We show in Section \ref{sec: Remarks} that the conclusion of Theorem \ref{mainthm2} is in fact false for $\n^*\n$ acting on vector fields.
\end{remark}

\begin{remark}
If $B = 0$ and $\n$ is metric with respect to $a$, then it is clear that $P$ is admissible.
The condition is however much more general than that. 
For example, assume that $\n$ is metric with respect to $a$, i.e.
\[
	\n a = 0,
\]
and assume that $B|_\H = \sum_{i = 1}^n e_i \otimes L_i$, where $e_i$ is some local frame in $T\H$ and $L_i$ are local endomorphism fields of $F$. 
Since $V$ is lightlike it follows that $g(V, e_i) = 0|_\H$ and hence equation \eqref{eq: admissible} is satisfied. 
\end{remark}

\textbf{Well-posedness for admissible scalar valued wave equations.}
Let us move on to the existence results. 
We consider scalar valued equations, i.e.\ $F = \R \times M$ with the canonical connection $\n := \d$ and $a$ given by pointwise multiplication.
Since the coefficients are real, complex scalar valued solutions can be treated as real valued solutions by considering the real and imaginary parts separately.
The usual d'Alembert operator is given by 
\[
	\Box := \n^*\n,
\]
or expressed in local coordinates as
\[
	\Box = - g^{\a \b}\left(\d_\a \d_\b - \d_{\n_{\d_\a}\d_\b}\right) = - g^{\a \b}\left(\d_\a \d_\b - \Gamma^\gamma_{\a\b}\d_\gamma\right).
\]
The general form of a wave operator acting on scalar valued functions is
\[
	P = \Box + \d_W + \a
\]
where $W$ is a smooth vector field and $\a$ is a smooth function. 
It turns out that we can express the admissibility condition in Definition \ref{def: admissible} as a condition on the restriction of $W$ to $\H$.
Recall that $\H \sqcup D(\S)$ is a smooth manifold with boundary $\H$.
We can ask for each $p \in \H$ whether $W|_p$ is inward pointing (pointing into $D(\S)$), tangent to $\H$ or outward pointing (pointing out of $D(\S)$). 

\begin{lemma}[Admissible linear scalar valued wave equations]\label{le: scalar_admissible}
A wave operator $P = \Box + \d_W + \a$ is admissible in the sense of Definition \ref{def: admissible} (with the above choices of $F$, $\n$ and $a$) if and only if $W|_\H$ is nowhere outward pointing, i.e.\ if $W|_p$ is inward pointing or tangent to $\H$ at each $p \in \H$.
\end{lemma}

\noindent 
If for example $W|_\H = 0$, then $P$ is admissible.
We give the proof of Lemma \ref{le: scalar_admissible} in Section \ref{sec: linear_characteristic}.

\begin{thm}[Well-posedness for admissible scalar valued wave equations]\label{mainthm3}
Assume that $\Ric(V, X) = 0$ for all $X \in T\H$.
Let $P = \Box + \d_W + \a$ be a wave operator on scalar valued functions such that
\[
	W|_\H
\]
is nowhere outward pointing.
For every $u_0 \in C^\infty(\H)$ and $f \in C^\infty(\H \sqcup D(\S))$, there is a unique $u \in C^\infty(\H \sqcup D(\S))$ such that
\begin{align*}
	Pu &= f, \\
	u|_{\H} &= u_0.
\end{align*}
Moreover, the solution $u$ depends continuously on the data $(u_0, f)$.
\end{thm}

\begin{remark}
A natural condition to impose on the Ricci curvature in general relativity is the dominant energy condition.
Defining $T := \Ric - \frac12 \mathrm{Scal} g$, it requires that for any future pointing causal vector field $X$, $-T(X, \cdot)^\sharp$ is a future pointing causal vector field as well.
Since $\H$ is totally geodesic, it follows by \cite{Kupeli1987}*{Thm. 30} that $g(\n_X V, Y) = 0$ for all tangent vectors $X, Y \in T\H$. 
Using this, one calculates that $\Ric(V, V)|_{\H} = 0$.
Now if the dominant energy condition is satisfied, then $T(V, V)|_{\H} = \Ric(V, V)|_{\H} = 0$.
Therefore $-T(V, \cdot)^\sharp|_\H$ is tangent to $\H$ and causal and hence lightlike.
Hence $-T(V, \cdot)^\sharp|_\H = fV$ for some smooth function $f$.
We conclude that $\Ric(V, X)|_{\H} = T(V, X)|_\H = g(-fV, X)|_\H = 0$ for any $X \in T\H$.
\end{remark}

\noindent 
It is interesting to note that in general, there exist solutions to $P u = 0$ on $D(\S)$ that do not extend continuously up to $\H$. 
See Remark \ref{rmk: blowup} for an explicit example. 
In Theorem \ref{mainthm3} we consider only those solutions that extend smoothly to the Cauchy horizon.
Let us also emphasise that we only assume specify the restriction of the function to $\H$ in Theorem \ref{mainthm3}. 
This is typical for characteristic Cauchy problems (where the initial hypersurface is lightlike), c.f. \cite{BaerWafo2014}*{Sec. 4}.
We show in Counter Example \ref{coex: scalar} that both the uniqueness and the existence statements in Theorem \ref{mainthm3} are false for general wave operators, i.e.\ if we drop the assumption on $W$. 
This is rather surprising, since both the Cauchy problem and the usual characteristic Cauchy problem on globally hyperbolic spacetimes are uniquely solvable for \emph{any} linear wave equation, see \cite{BaerWafo2014}. 
In particular, the well-posedness theory in our setting depends not on the principal symbol (which is fixed for wave operators), but on the \emph{first order part} of the wave operator!

\textbf{Local well-posedness for admissible scalar valued semi-linear wave equations.}
Using Theorem \ref{mainthm3} and the energy estimates, we are able to show local existence for semi-linear wave equations. 

\begin{thm}[Local well-posedness for admissible scalar valued semi-linear wave equations]\label{mainthm4}
Assume that $\Ric(V, X) = 0$ for all $X \in T\H$.
Let $P = \Box + \d_W + \a$ be a wave operator such that
\[
	W|_\H
\]
is nowhere outward pointing. 
Let $f \in C^\infty((\H \sqcup D(\S))\times \R)$ and let $u_0 \in C^\infty(\H)$ be given. 
Then there is an open subset $U \subset \H \sqcup D(\S)$ such that $\H \subset U$ and a unique $u \in C^\infty(U)$ such that
\begin{align*}
	P u &= f(u), \\
	u|_\H &= u_0. 
\end{align*}
\end{thm}

\noindent We believe that one is able to prove local existence for a much larger class of non-linear wave equations than those treated in Theorem \ref{mainthm4} using the methods presented here. 
For simplicity of presentation, we only consider the formulation in Theorem \ref{mainthm4}.

\begin{remark}
We show in Section \ref{sec: Remarks}, by giving explicit counter examples, that Theorem \ref{mainthm1}, Theorem \ref{mainthm2}, Theorem \ref{mainthm3} and Theorem \ref{mainthm4} were all false if we allowed for vanishing surface gravity.
\end{remark}

\textbf{Application to the strong cosmic censorship conjecture.}
Let us now discuss the already mentioned connection to general relativity in more detail.
By the strong cosmic censorship conjecture, it is expected that vacuum spacetimes containing a compact Cauchy horizon have very special geometry.
The study of this problem was initiated by Moncrief in \cite{Moncrief1982}, \cite{Moncrief1984} and by Moncrief and Isenberg in \cite{MoncriefIsenberg1983}, \cite{IsenbergMoncrief1985}.
Most importantly, they made the following conjecture in $1983$:
\begin{conjecture}[Moncrief-Isenberg \cite{MoncriefIsenberg1983}] \label{conj: Moncrief-Isenberg}
Any smooth vacuum spacetime with a compact Cauchy horizon $\H$ admits a Killing vector field in the globally hyperbolic region $D(\S)$, extending smoothly up to the Cauchy horizon $\H$.
\end{conjecture}
\noindent 
Moncrief and Isenberg proved this conjecture in \cite{MoncriefIsenberg1983} in spacetime dimension $4$ under the assumptions that the spacetime is analytic and that the generators of the Cauchy horizon are closed.
They moreover proved that if the generators are closed, then the surface gravity can always be normalised to a constant (however potentially zero).
The same authors recently generalised their result in \cite{MoncriefIsenberg2018} to also cover the case when the generators of the horizon densely fill a $2$-torus (the analytic non-ergodic case), still under the assumption of analyticity. 
In both these results, they used on analyticity of the spacetime metric and a Cauchy-Kovalewski argument to propagate the Killing vector field off the Cauchy horizon.

Friedrich, Rácz and Wald proved the conjecture for smooth spacetimes of dimension $4$ under the assumptions that the compact Cauchy horizon has closed generators and that the surface gravity can be normalised to a non-zero constant in \cite{FRW1999} (see also \cite{R2000} including matter fields and \cite{HollandsIshibashiWald2007} \cite{IsenbergMoncrief2008} for extensions to higher dimensions).
They relied on the fact that the generators of the Cauchy horizon were closed and that the initial data was invariant along the generators.
This allowed to transform the problem into the well studied characteristic Cauchy problem with initial data on two intersecting null hypersurfaces.

If one does not assume that the generators close or that the spacetime metric is analytic, one instead needs Theorem \ref{mainthm1} to solve this wave equation.
In order to apply Theorem \ref{mainthm1} to Conjecture \ref{conj: Moncrief-Isenberg}, the author and István Rácz generalise in \cite{PetersenRacz2018} the computations by Moncrief-Isenberg in \cite{MoncriefIsenberg1983}.
We prove that the Killing equation can be solved up to any order at a compact Cauchy horizon with constant non-zero surface gravity, without any assumption on the generators.
Theorem \ref{mainthm1} can then be applied to prove the following statement (which holds in any spacetime dimension $n+1 \geq 2$).

\begin{thm}[see \cite{PetersenRacz2018}] \label{thm: PetersenRacz}
Any smooth vacuum spacetime containing a compact Cauchy horizon $\H$, with surface gravity that can be normalised to a non-zero constant, admits a Killing vector field in the globally hyperbolic region $D(\S)$, extending smoothly up to the Cauchy horizon $\H$.
\end{thm}
\noindent 
Again, to the best of our knowledge, all known examples of compact Cauchy horizons in vacuum spacetimes satisfy our assumption on the surface gravity.
As a consequence of Theorem \ref{thm: PetersenRacz}, the maximal globally hyperbolic development of \emph{generic} compact or asymptotically flat vacuum initial data 
%that is not Killing initial data (see \cite{BCS2005} and \cite{Moncrief1973}*{Thm. 6.1}) 
cannot be extended over a compact Cauchy horizons with surface gravity that can be normalised to a non-zero constant. 
Theorem \ref{thm: PetersenRacz} constitutes therefore a natural first step towards the strong cosmic censorship conjecture without symmetry assumptions.

A topic related to this work is to understand the asymptotics of the Einstein equation, using methods for so called Fuchsian equations. 
The goal there is to prescribe the asymptotic behaviour at the initial singularity and show the existence of a solution to the Einstein equation with those asymptotics. 
Solutions with asymptotics that are bounded towards the initial singularity could be interpreted as solutions of certain wave equations with initial data on a compact Cauchy horizon. 
Most of the results are done in the analytic setting, see \cite{AnderssonRendall2001} and references therein. 
Some more recent results dropped the assumption of analyticity, see \cite{ABIL2013}, \cite{ABIL2013_2}, \cite{BH2012}, \cite{BH2014}, \cite{Rendall2000}, \cite{BL2010} and \cite{Stahl2002} and references therein. 
In all cases known to the author, where analyticity is not assumed, one instead assumes symmetry of the spacetime when applying the Fuchsian methods. 
In contrast to this, we neither assume analyticity nor symmetry of the spacetime. 

There are many results on the characteristic Cauchy problem, also for non-linear wave equations, when the lightlike hypersurface is a light cone or two intersecting lightlike hypersurfaces, see \cite{CP2012}, \cite{DossaTadmon2010} and \cite{Rendall1990} and references therein. 
In these results, the solution is shown to exist on (a part of) the domain of dependence of the lightlike hypersurface where the initial data are specified.
Since the domain of dependence of a compact Cauchy horizon is in general nothing but the Cauchy horizon itself (see Remark \ref{rmk: domain_dependence}), these results do not apply to our case.
For linear wave equations, the characteristic Cauchy problem has been studied for more general lightlike hypersurfaces by Bär and Tagne Wafo in the already mentioned \cite{BaerWafo2014}, generalising results of Hörmander in \cite{Hormander1990}. 
The difference to our work is that they consider lightlike initial hypersurfaces which are subsets of globally hyperbolic spacetimes, whereas we do not.

The paper is structured as follows. 
In Section \ref{sec: Remarks} we give examples of spacetimes containing compact Cauchy horizons and give counter examples of seemingly potential generalisations of our main results.
Section \ref{sec: Energy} is the analytic core of the paper, where we formulate and prove our energy estimate. 
In Sections \ref{sec: global_given_asymptotic}-\ref{sec: non-linear_characteristic} we prove our main results, Theorems \ref{mainthm1}, \ref{mainthm2}, \ref{mainthm3} and \ref{mainthm4}.

\subsubsection*{Acknowledgements}
I am especially grateful to Andreas Hermann for numerous discussions and helpful comments on this work. 
I also want to thank my PhD supervisor Christian Bär for important discussions concerning initial value problems for wave equations on curved spacetimes. 
Furthermore, I want to thank Vincent Moncrief, István Rácz, Hans Ringström and Florian Hanisch for helpful comments. 
Finally, I would like to thank the Berlin Mathematical School, Sonderforschungsbereich 647 and Schwerpunktprogramm 2026, funded by Deutsche Forschungsgemeinschaft, for financial support.

\section{Examples and remarks} \label{sec: Remarks}

Before proceeding with the proofs, let us start by giving counter examples to three seemingly potential generalisations of our main results. 
We also give some examples to point out certain differences between the problem studied here and the characteristic Cauchy problem in globally hyperbolic spacetimes \cite{BaerWafo2014}.
Most importantly, due to the peculiar nature of compact Cauchy horizons, there will be no use of the fact that solutions to wave equations obey \enquote{finite speed of propagation} (see Remark \ref{eq: no_finite_speed}).
The behaviour of waves close to Cauchy horizons is naturally \enquote{global}.
This section is independent of the proofs of our main results, which starts in Section \ref{sec: Energy}.
We first study the simplest spacetime where our results apply.

\begin{example}[Misner spacetime] \label{ex: Misner spacetime}
Define the Misner spacetimes as
\begin{equation*} \label{Misner space}
	(M_\pm, g) := (\R \times S^1, \pm 2dtdx + tdx^2).
\end{equation*}
If we choose $\S := \{C\} \times S^1$ in $M_\pm$ for $C > 0$, the past Cauchy horizon is given by $H_-(\S) = \{0\} \times S^1$ and the future Cauchy horizon is empty. 
The Misner spacetime $M_+$ is illustrated in Figure \ref{fig: Misner}.
It is clear that $H_-(\S)$ is totally geodesic and we claim that the surface gravity of $H_-(\S)$ can be normalised to a non-zero constant. 
Choosing $V := \d_x$, one calculates that 
\begin{align*}
	g(\n_V V, \d_t) &= g(\n_{\d_x}(\d_x), \d_t) = - g(\d_x, \n_{\d_x} \d_t) = - \frac12 \d_t g(\d_x, \d_x) = - \frac12, \\
	g(\n_V V, \d_x) &= g(\n_{\d_x}(\d_x), \d_x) = \frac12 \d_x g(\d_x, \d_x) = 0,
\end{align*}
which implies that $\n_V V = \mp \frac12 V$ on $M_\pm$. 
This show that the surface gravity is constant and non-zero. 
Therefore our main results apply with $\H := H_-(\S)$.
\end{example}

Let us now describe three peculiar features of the characteristic Cauchy problem for initial data on compact Cauchy horizons that differs strongly from usual Cauchy problem.

\begin{remark}[No use of finite speed of propagation] \label{eq: no_finite_speed}
The d'Alembert operator on the Misner spacetime $M_+$ is $\Box = \d_t(t\d_t - 2\d_x)$. Consider the admissible wave equation
\begin{equation} \label{eq: no_finite_speed_eq}
	\Box u + u = 0.
\end{equation}
Given any initial data $u_0 \in C^\infty(\{0\} \times S^1)$, Theorem \ref{mainthm3} implies that there is  $u \in C^\infty([0, \infty) \times S^1)$ solving \eqref{eq: no_finite_speed_eq} such that $u|_{t = 0} = u_0$. 
Evaluating \eqref{eq: no_finite_speed_eq} at $t = 0$ gives $- 2\d_x(\d_tu|_{t = 0}) + \d_tu|_{t = 0} + u_0 = 0$. 
This implies that $\d_tu|_{t = 0}(x) = - \int_{-\infty}^0 e^s u_0(x - 2s)ds$, where we now consider $u_0$ as a periodic function on $\R$.
If $u_0 \geq 0$ and not identically zero, it follows that $\d_tu|_{t = 0}(x) < 0$ for \emph{every} $x \in S^1$! 
By continuity of $u$, we conclude that there is a $\delta > 0$ such that $(0, \delta) \times S^1 \subset \supp(u)$!
This holds, even if $u_0$ is identically zero on a piece of $S^1$. 
Strictly speaking, this does not contradict finite speed of propagation, since for any point $p \in \{0\} \times S^1$, the causal future is $J_+(p) = [0, \infty) \times S^1$. 
However, it shows that the finite speed of propagation does not tell us anything if initial data is specified on the Cauchy horizon. 
Proving our main results is therefore a \emph{non-local} problem and cannot be studied locally on one coordinate patch at the time.
\end{remark}

\begin{remark}[Solutions may blow up at the Cauchy horizon] \label{rmk: blowup}
Note that $u(t, \cdot) = \ln(t)$ satisfies $\Box u = 0$ on the Misner spacetimes $M_\pm$. 
Moreover, $u$ does not extend continuously to the horizon $\H = \{0\} \times S^1$. 
We conclude that there are solutions to $\Box u = 0$, defined on $D(\S)$ that \enquote{blow up} at the Cauchy horizon $\H$.
\end{remark}

\begin{remark}[Domain of dependence of a Cauchy horizon] \label{rmk: domain_dependence}
Consider again the Misner spacetime $M_+$. 
The vector field $-\frac t2\d_t + \d_x$ is lightlike and its integral curves will \enquote{spiral} around $\{0\} \times S^1$ for $t>0$ and $t < 0$ close to $0$ but never intersect $\{0\} \times \S$, compare with the light cones illustrated in Figure \ref{fig: Misner}. 
This implies that the domain of dependence of the Cauchy horizon $\{t_0\} \times \S$ is nothing but the Cauchy horizon itself. 
Therefore our main results control the solution \emph{outside the domain of dependence} of the Cauchy horizon where initial data is prescribed.
\end{remark}

The following counter examples also illustrate the difference between the problem studied here and the usual characteristic Cauchy problem.

\begin{counterexample}[Thm. \ref{mainthm2}, Thm. \ref{mainthm3} and Thm. \ref{mainthm4} fail for general wave operators]
 \label{coex: scalar}
Let us give two examples to show that neither existence nor uniqueness holds in Theorem \ref{mainthm3} and Theorem \ref{mainthm4} without the admissibility assumption. 
Consequently, also Theorem \ref{mainthm2} is false without the admissiblity assumption.
The d'Alembert operator on the Misner spacetime $M_+$ is given by
\[
	\Box = t\d_t\d_t - 2 \d_x\d_t + \d_t.
\]
\begin{itemize}
\item First consider the non-admissible operator $P:= \Box - \d_t$. Note that $u(t,x) = Ct$ solves $P u = 0$ for all $C\in \R$ and $u|_{\H_-} = 0$. We conclude that uniqueness does not hold for all wave operators.
\item Now consider the non-admissible operator $P := \Box - \d_t + 1$. Assume that $Pu = 0$ and let $u_0 := u|_{t = 0}$. 
Then $Pu|_\H = -2 \d_x \d_t u|_\H + u_0 = 0$.
But if we integrate this equation over $S^1$, we conclude that $\int_{S^1} u_0(s)ds = 0$, which is a strong restriction on the initial data. 
If $u_0$ is \emph{not} satisfying this restriction, there is no solution $u$ to $P u = 0$ with $u|_{t = 0} = u_0$. 
The conclusion is that existence of solution does not hold for all wave operators.
\end{itemize}
A natural wave operator is $\n^*\n$ acting on tensors, where $\n$ is the Levi-Civita connection with respect to $g$.
We already remarked in the introduction that we cannot choose $a = g$ in Defintion \ref{def: admissible}, since $g$ is not positive definite.
In fact, we show here that the conclusion of Theorem \ref{mainthm2} is false for this operator.
Consider the Misner spacetime $M_+$ and let $\n$ be the Levi-Civita connection on vector fields.
Using that $\n_{\d_x}\d_x = - \frac12\d_x + \frac12 t \d_t$, $\n_{\d_t}\d_x = \frac12 \d_t$ and $\n_{\d_t}\d_t = 0$, it follows that
\[
	\n^*\n = t\n_{\d_t}\n_{\d_t} - 2 \n_{\d_x} \n_{\d_t} + \n_{\d_t}
\]
and hence that
\[
	\n^*\n (t \d_t) = 0.
\]
Since $t\d_t|_{t = 0} = 0$, we have found a nontrivial solution to trivial initial data, despite the simple geometric nature of the operator $\n^*\n$. 
This shows that Theorem \ref{mainthm2} were false if we dropped the assumption that the wave operator is admissible in the sense of Definition \ref{def: admissible}.

\end{counterexample}

\noindent In the next example we modify the Misner spacetime in order to show that all four main results are wrong for compact Cauchy horizons with vanishing surface gravity.

\begin{counterexample}[All four main results fail for zero surface gravity] \label{counterex: non-zero}
Consider the spacetime 
\[
	(M, g) := (\R \times S^1, 2dtdx + t^4 dx^2)
\]
Let $\S := \{1\} \times S^1$. 
By a calculation similar to that for the Misner spacetime, it follows that $H_-(\S) = \{0\} \times S^1$ is smooth, compact and totally geodesic and the surface gravity vanishes. 
It is easy to check that if the surface gravity vanishes, one cannot normalise it to a non-zero constant.
The d'Alembert operator is given by $\Box = \d_t(t^4 \d_t - 2 \d_x)$.
Define $u \in C^\infty(M)$ by
\[
	u(t, x) := \begin{cases} 	e^{ - \frac1t}, \text{ if }t > 0, \\
						0, \text{ if } t \leq 0.
            \end{cases}
\]
It follows that 
\[
	\Box u - (2t +1)u = 0
\]
and $(\d_t)^n u|_{H_-(\S)} = 0$ for all $n \in \N$. 
This shows that we cannot drop the non-zero surface gravity assumption in neither Theorem \ref{mainthm1}, Theorem \ref{mainthm2}, Theorem \ref{mainthm3} nor Theorem \ref{mainthm4}.

It is of interest to note that also the existence statement fails in Theorem \ref{mainthm3} and Theorem \ref{mainthm4} even for the equation $\Box u = 0$ if one drops the non-zero surface gravity assumption. 
Consider the spacetime 
\[
	(M, g) := (\R \times (S^1)^2, 2dtdx + t^m dx^2 + dy^2)
\]
for an integer $m > 1$. 
Let $\S := \{1\} \times (S^1)^2$. 
Again, $H_-(\S) = \{0\} \times S^1$ is smooth, compact and totally geodesic and has vanishing surface gravity. 
The d'Alembert operator is given by $\Box = \d_t(t^m \d_t - 2 \d_x) - \d_y^2$.
Assume now that $\Box u = 0$ and let $u_0 := u|_{t = 0}$.
It follows that
\[
	\d_x(\d_t u(0, x,y)) + \d_y^2u_0 = 0,
\]
for all $(x, y) \in (S^1)^2$.
Integrating over $S^1$ in the $x$-variable gives
\[
	\d_y^2 \int_{S^1} u_0(0, s, y) ds = 0.
\]
This means that $y \mapsto \int_{S^1} u_0(0, s, y) ds$ is constant, which is a strong restriction on our choice of initial data $u_0$. 
This shows that also existence of solution fails in Theorem \ref{mainthm3} and Theorem \ref{mainthm4} if we drop the non-zero surface gravity assumption. 
\end{counterexample}

An important example of a \emph{vacuum} spacetime containing two compact Cauchy horizons with non-zero constant surface gravity is the Taub-NUT spacetime.

\begin{example}[The Taub-NUT-spacetime] \label{ex: Taub-NUT}
The Taub-NUT spacetime is defined by
\[
	(M, g) := (\R \times S^3, \pm 4 l dt \sigma_1 + 4 l^2 U(t) {\sigma_1}^2 + (t^2 + l^2)({\sigma_2}^2 + {\sigma_3}^2)),
\]
where 
\begin{align*}
	U(t) := \frac{(t_+-t)(t-t_-)}{t^2 + l^2}
\end{align*}
and where $t_\pm:= m \pm \sqrt{m^2+l^2}$, with $m \in \R$, $l > 0$ and $\sigma_1, \sigma_2, \sigma_3$ are orthonormal left invariant one-forms on $S^3$. 
Let us choose $\S := \{\tau \} \times S^3$ for some $\tau \in (t_-, t_+)$. Then $\S$ is an acausal closed hypersurface. 
The past and future Cauchy horizons of any such $\S$ are given by
\begin{align*}
	H_-(\S) &= \{t_-\} \times S^3, \\
	H_+(\S) &= \{t_+\} \times S^3,
\end{align*}
which are clearly compact. 
Since $(M, g)$ is Ricci flat, Remark \ref{rmk: null_energy_condition} implies that $H_-(\S)$ and $H_+(\S)$ are totally geodesic. 
Similarly to the Misner spacetimes, one calculates that both Cauchy horizons have constant non-zero surface gravity. 
Our main results therefore apply to the Taub-NUT spacetime with $\H := H_-(\S)$ or $\H := H_+(\S)$. 
\end{example}

Let us also illustrate an example where our main results apply and the generators of the Cauchy horizon do not close. 
This is interesting because the techniques of \cite{FRW1999} and \cite{R2000}, where certain wave equations are solved for initial data on compact Cauchy horizons, rely on the fact that the generators are closed (and on the fact that the initial data is invariant along the generators).
Our results apply without any assumptions on the generators, hence also to the next example.

\begin{example}[A compact Cauchy horizon with non-closed generators] \label{ex: non-closed_generators}
Consider the spacetime $M := \R \times \R^n$ with metric
\[
	g = 2dtdx^1 + t (dx^1)^2 + \sum_{j=2}^n (dx^j)^2.
\]
It is easy to compute that $g$ is flat.
Let $\Gamma$ be a grid on $\R^n$. 
By translation invariant of the metric, the metric is induced on the quotient $\R \times \R^n/\Gamma$.
If the quotient is compact, then $\H := \{t = 0\}$ is a totally geodesic compact smooth past Cauchy horizon.
By choosing $\Gamma$ with \enquote{irrational angles} one can produce examples such that any integral curve of the lighlike vector field $V := \d_{x^1}$ densely fill $\H$.
By the calculation in Example \ref{ex: Misner spacetime}, we know that
\[
	\n_V V = - \frac12 V,
\]
so the surface gravity is normalised to a non-zero constant.
\end{example}
\noindent
In \cite{MoncriefIsenberg2018}, Moncrief and Isenberg treat the case when the generators densely fill a 2-torus. 
From the above example it is clear that more general behaviour of the generators is possible. 
Example \ref{ex: non-closed_generators} satisfies however the assumptions of this paper and of \cite{PetersenRacz2018}.

\section{An energy estimate close to the Cauchy horizon} \label{sec: Energy}

Let us from now assume that $\H$ is a \emph{past} Cauchy horizon, in addition to Assumption \ref{as: Cauchy_horizon}.
The case when $\H$ is a future Cauchy horizon is obtained by a time reversal.
In this section we state and prove our energy estimate.

\subsection{A null time function} \label{sec: null time function}

The first step towards formulating the energy estimate is to express the metric in terms of a \enquote{null time function} on a small future neighbourhood of $\H$.
This is reminiscent of the locally defined \enquote{Gaussian null coordinates} used by Moncrief and Isenberg in \cite{MoncriefIsenberg1983}. 
As we have seen in Remark \ref{eq: no_finite_speed}, to solve wave equations with initial data on a compact Cauchy horizon is genuinely a \emph{non-local} problem. 
It is therefore convenient to avoid working with local coordinates.
Recall that $V$ denotes the nowhere vanishing lightlike vector field tangent to $\H$ such that 
\[
	\n_V V = V
\]
on $\H$.
Since $\H$ is totally geodesic, \cite{Kupeli1987}*{Thm. 30} implies that 
\[
	g(\n_X V, Y) = 0
\]
for all $X, Y \in T\H$. 
Therefore there is a smooth one-form $\o$ on $\H$ such that
\[
	\n_X V = \o(X) V
\]
for all $X \in T\H$.
Note that $\o(V) V = \n_V V = V$ and therefore $\omega(V) = 1$. 
Since $\o$ is nowhere vanishing, $E := \ker(\o)$ is a sub-vector bundle of $T\H$.
We get the splitting
\[
	T\H = \R V \oplus E.
\]
We will construct a null time function by flowing $\H$ along the geodesics emanating from a certain lightlike vector field transversal to $\H$.
This will give a diffeomorphism of a small future neighbourhood of $\H$ in $\H \sqcup D(\S)$ to $[0, \e) \times \H$.
Here $\H$ will correspond to $\H \times \{0\} \times \H$.
Let us therefore consider $E$ as a subbundle of $T(\{t\} \times \H)$ in the canonical way.

\begin{prop}[The null time function] \label{prop: time_function}
There is an open neighbourhood $U \subset \H \sqcup D(\S)$, containing $\H$ and a smooth function $t:U \to \R$ such that $(U, g|_U)$ is isometric to
\[
	[0, \epsilon) \times \H,
\]
where $t$ is the coordinate on $[0, \epsilon)$ and the metric takes the form
\begin{equation}\label{eq: the metric}
 	\begin{pmatrix}
 		0 & 1 & 0 \\
 		1 & - \psi & 0 \\
 		0 & 0 & \g
 	\end{pmatrix}
\end{equation}
with respect to the splitting $T([0, \epsilon) \times \H) = \R\d_t \oplus \R\grad(t) \oplus E$. 
Here $\psi \in C^\infty([0, \epsilon) \times \H)$ is such that $\d_t\psi(t, \cdot), \psi(t, \cdot) > 0$ for all $t \in (0, \epsilon)$ and $\d_t \psi(0, \cdot) = 2$ and $\psi(0, \cdot) = 0$. 
The induced (time-dependent) metric $\g$ on the vector bundle $E$ is positive definite. 
Moreover, $\d_t$ is a lightlike geodesic vector field and $\{t\} \times \H$ are spacelike Cauchy hypersurfaces in $D(\S)$ for any $t \in (0, \epsilon)$.
\end{prop}

\noindent
Note that the null time function that we have constructed is a natural generalisation of the null time function that comes with the Misner spacetime.
In other words, the proposition establishes a foliation reminiscent of Figure \ref{fig: Misner} for any Cauchy horizon satisfying our assumptions.

Before proving Proposition \ref{prop: time_function}, let us first discuss the statement in more detail.
Since $\d_t$ is lightlike, and $g(\grad(t), \d_t) = 1$, it follows that $\d_t$ and $\grad(t)$ are linearly independent.
This explains the splitting $T([0, \epsilon) \times \H) = \R\d_t \oplus \R\grad(t) \oplus E$.
Define 
\[
	\H_t  := \{t\} \times \H \subset [0, \e) \times \H,
\]
for all $t \in [0, \e)$. 
Proposition \ref{prop: time_function} gives a canonical diffeomorphism between the Cauchy horizon $\H = \H_0$ and the Cauchy hypersurfaces $\H_t$ for $t \in (0, \epsilon)$. 
As already explained, we consider $E \subset T\H_t$ as a subbundle in the canonical way. 
We can also extend $V$ as a smooth vector field with $V \in T\H_t$ for all $t \in [0, \e)$ by demanding that $[V, \d_t] = 0$.
It follows that 
\[
	T\H_t = \R V \otimes E|_{\H_t}
\]
for all $t \in [0, \e)$.
For the proof of the energy estimate, another way to decompose $T\H_t$ turns out to be useful. 
By Proposition \ref{prop: time_function}, it follows that $-\psi = g(\grad(t), \grad(t))$, which implies that $\grad(t) = Z - \psi\d_t$, for a certain smooth vector field $Z$ such that $Z|_{\H_t} \in T\H_t$ for all $t \in [0, \epsilon)$ and $Z|_{\H} = -V$. 
Let from now on $Z$ denote this vector field. 
It follows that
\begin{equation*}
	T\H_t = \R Z \otimes E|_{\H_t}
\end{equation*}
for all $t \in [0, \e)$ after shrinking $\e$ if necessary. 
We now prove Proposition \ref{prop: time_function}.

\begin{proof}
Since $\n_V V = V$, one can reparametrize the integral curves of $V$ to geodesics tangent to $V$.
The geodesics are complete in the positive direction and incomplete in the negative direction. 
Therefore for example \cite{Larsson2014}*{Lem. 1.6} implies that $V$ is future directed.
Since $E \subset TM|_\H$ is a Riemannian subbundle, it follows that $T M|_\H = E \oplus E^\perp$. 
Now, since $M$ is time-oriented, there is a nowhere vanishing timelike vector field $T$ on $M$. 
Since $\H$ is lightlike, $T|_\H$ must be transversal to $\H$.
Let us now project $T|_\H$ onto $E^\perp$, we call the projection $\bar T|_\H$.
$\bar T|_\H$ must be nowhere vanishing and everywhere linearly independent of $V$, for otherwise we get a contradiction to the fact that $T|_\H$ is transversal to $\H$.
Since $E^\perp$ is a subbundle of rank $2$ and $V \in E^ \perp$, this shows that $E^\perp$ is a trivial bundle spanned by $\tilde T|_\H$ and $V$.
Since $E^\perp$ is a Lorentzian subbundle, we conclude, there is a unique future pointing nowhere vanishing lightlike vector field $L \in E^\perp$ such that $g(L, V) = - 1$. 

Consider now the map
\begin{align*}
	f_s: \H &\to M, \\*
	x &\mapsto \exp_x(Ls),
\end{align*}
for those $s \in \R$ where this map is defined. 
By compactness of $\H$, there is an $\epsilon > 0$ such that $f_s$ is defined for all $s \in [0, \epsilon)$ and such that the map
\begin{align*}
	\hat f: [0, \epsilon) \times \H &\to M, \\
	(s, x) &\mapsto f_s(x),
\end{align*}
is a diffeomorphism (of manifolds with boundary) onto its image. 
Let us show that we can shrink $\e$ such that $\hat f((0, \epsilon) \times \H) \subset D(\S)$. 
For each $x \in \H$, $s \mapsto f_s(x)$ is a future pointing lightlike curve. Since $\H \sqcup D(\S)$ is a smooth manifold with boundary $\H$ and $L$ is transversal on $\H$, either $f_s(x) \in D(\S)$ for all negative $s$ close to zero or for all positive $s$ close to zero. 
Assume to reach a contradiction that $f_s(x) \in D(\S)$ for all negative $s$ close to $0$. 
It follows by continuity that there exists a past directed timelike curve $\gamma$ such that $\gamma(0) = x$ and $\gamma(s') \in D(\S)$ for all negative $s'$ close to $0$. 
Since $D(\S)$ is globally hyperbolic, we can extend $\gamma$ to a past directed timelike curve reaching $\S$. 
In other words, $x \in I_+(\S)$. 
On the other hand, by \cite{O'Neill1983}*{Prop. 53 (1)}, $x \in \H \subset I_-(\S)$. 
This contradicts acausality of $\S$, and we conclude that $f_s(x) \in D(\S)$ for small positive $s$. 
It follows by compactness of $\H$ that we can shrink $\e$ so that $\hat f((0, \e) \times \H) \subset D(\S)$. 
Define $U := \hat f([0, \epsilon) \times \H)$. 

Denote the coordinate on $[0, \epsilon)$ by $t$. 
By construction we have $\n_{\d_t} \d_t = 0$, i.e.\ $\d_t$ is a geodesic vector field. 
Since $\d_t|_\H = L$, $\d_t$ is a lightlike vector field.
Let now $X \in E \subset T\H$ and extend $X$ to $T([0, \epsilon) \times \H)$ by $(0, X)$, which we still denote by $X$. 
It follows that $[X, \d_t] = 0$. 
Moreover, it follows that $\d_t g(\d_t, X) = g(\n_{\d_t}\d_t, X) + g(\d_t, \n_{\d_t}X) = g(\d_t, \n_X\d_t) = \frac12 \d_X g(\d_t, \d_t) = 0$. 
Since $g(\d_t, X)|_{\{0\} \times \H} = g(L, X) = 0$, we conclude that $g(\d_t, X) = 0$ everywhere. 
Since $\grad(t)|_{\{0\} \times \H} = - V$, we note that $\d_t$ and $\grad(t)$ are linearly independent on $[0, \epsilon) \times \H$, making $\epsilon$ even smaller if necessary. 
Moreover, $g(\grad(t), X) = dt(X) = \d_Xt = 0$ and $g(\grad(t), \d_t) = 1$ everywhere. 
This implies that $X, \d_t$ and $\grad(t)$ are everywhere linearly independent.
We have shown up to now that
\begin{align*}
	g(\d_t, \d_t) &= 0,  \\
	g(\d_t, \grad(t)) &= 1, \\
	g(\d_t, X) &= 0, \\
	g(\grad(t), X) &= 0.
\end{align*}
We define $\psi := -g(\grad(t), \grad(t))$, which completes the form of the metric stated in \eqref{eq: the metric}. 
Since $\grad(t)|_{\{0\}\times \H} = -V$ it follows that $\psi(0, \cdot) = - g(-V, -V) = 0$. 
In order to calculate $\d_t \psi(0, \cdot)$, first extend the vector field $V$ to $T([0, \epsilon) \times \H)$ by $(0, V)$, still denoting it $V$. 
We have $[V, \d_t] = 0$. It follows that $\d_t \psi(0, \cdot) = - 2 g(\n_{\d_t}\grad(t), - V)|_{\{0\} \times \H} = - 2\d_t g(\grad(t), -V)|_{\{0\} \times \H} + 2 g(-V, - \n_{\d_t} V)|_{\{0\} \times \H} =  2 g(V, \n_V \d_t)|_{\{0\} \times \H} = - 2 g( \n_V V, \partial_t)|_{\{0\} \times \H} = 2$. 
Shrinking $\epsilon$ if necessary, we can make sure that $\d_t \psi(t, \cdot) > 0$ for all $t \in [0, \epsilon)$ and $\psi(t, \cdot) > 0$ for all $t \in (0, \epsilon)$. 

Since $\psi > 0$, we have made sure that $\grad(t)$ is timelike on $(0, \epsilon) \times \H$, which implies that hypersurfaces $\{t\} \times \H$ are compact spacelike hypersurfaces in the globally hyperbolic spacetime $D(\S)$, for all $t \in (0, \epsilon)$. 
By \cite{BILY1978}*{Thm. 1} (the statement is given in $n = 3$, but the proof goes through in any dimension) it follows that the level sets $\{t\} \times \H$ are Cauchy hypersurfaces, for all $t \in (0, \epsilon)$.
\end{proof}

\subsection{Stating the energy estimate} \label{sec: Energy_estimate}

For each $s \in \R$, let $\norm{\cdot}_{s}$ denote a Sobolev norm for the Sobolev space $H^s(\H)$. 
Since $\H$ is a compact smooth manfold, all Sobolev norms are equivalent. 
We are formulating the energy estimate close to $\H$ in terms of the null time function $t$ given by Proposition \ref{prop: time_function}.
We use the notational convention $\N := \{0,1, \hdots \}$. 
Let from now on $\n_t := \n_{\d_t}$.

\begin{thm} \label{thm: Energy1}
Let $P$ be a wave operator. For any $m \in \N$, there exists a constant $D_m > 0$ such that for all $[t_0, t_1] \subset (0, \epsilon)$, we have
\begin{align*}
	\norm{u(t_1, \cdot)}_{2m+1} + \sqrt{t_1}\norm{\n_t u(t_1, \cdot)}_{2m} &\leq D_m \left(\frac{{t_1}}{{t_0}}\right)^{D_m}\left(\norm{u(t_0, \cdot)}_{2m+1} + \sqrt{t_0}\norm{\n_t u(t_0, \cdot)}_{2m}\right) \\*
	& \quad + D_m{t_1}^{D_m}\int_{t_0}^{t_1}\frac{\norm{P u(t, \cdot)}_{2m}}{t^{D_m + 1/2}}dt
\end{align*}
for all $u \in C^\infty([t_0, t_1] \times \H, F)$.
\end{thm}

\noindent Note that we are not allowed to put $t_0 = 0$ in Theorem \ref{thm: Energy1}.

\begin{remark}
When restricting to $t_0 \geq \delta > 0$, our energy estimate is essentially equivalent to the classical ones, like for example \cite{BaerWafo2014}*{Thm. 8}. 
\end{remark}

\begin{remark}
The natural \enquote{energy} in Theorem \ref{thm: Energy1} has the form 
\[
	\norm{u(t, \cdot)}_{2m+1} + \sqrt{t}\norm{\n_t u(t, \cdot)}_{2m}.
\] 
This means that the energy can control the value of the function at $t = 0$ but \emph{not} the first time derivative. 
This is actually what one would expect, since we only specify the value and not the first derivative at the Cauchy horizon in Theorem \ref{mainthm2}, Theorem \ref{mainthm3} and Theorem \ref{mainthm4}. 
\end{remark}

\subsection{Preparations for the proof}
\subsubsection{The Sobolev spaces}

One main problem in proving energy estimates for wave equations close to a Cauchy horizon, is that the horizon is lightlike. 
The metric degenerates on the horizon and there is no natural Sobolev norm coming from the geometry. 
We overcome this by introducing a certain Riemannian metric on $\H$. 
For each $p \in \H$ and $X, Y \in T_p\H$, we define
\begin{equation}
	\sigma(X, Y) := g(X, Y) + g(X, \d_t) g(Y, \d_t). \label{eq: sigma}
\end{equation}
It follows immediately that $\sigma|_{E \times E} = \g|_\H$. 
Since $\g|_\H$ is positive definite, $\sigma(V, V) = 1$ and ${\sigma(X, V) = 0}$ for all $X \in E$, we conclude that $\sigma$ is a Riemannian metric on $T\H$. 
By using the diffeomorphism $\H \cong \{t\} \times \H = \H_t$ given by Proposition \ref{prop: time_function}, we may consider $\sigma$ as a Riemannian metric on $\H_t$ for each $t \in [0, \e)$.

Let us denote the Levi-Civita connection with respect to $\sigma$ by $\hat \n$.
Using the connection $\n$ on $F$, we define the connection-Laplace operator $\Delta$ on $\H_t$ with respect to $\sigma$ by
\[
	\Delta h := - \sum_{i, j=1}^n \sigma^{ij}\left( \n_{e_i}\n_{e_j} - \n_{\hat \n_{e_i}e_j} \right)h,
\]
where $h \in C^\infty(\H_t, F|_{\H_t})$ and $e_1, \hdots, e_n$ is some local frame of $T\H_t$ for some fixed $t \in [0, \e)$. 

Theorem \ref{thm: Energy1} is formulated in terms of any Sobolev norm, but it will be convenient to use certain Sobolev norms that are constructed from $\sigma$.
For any $f_1, f_2 \in C^\infty([0, \epsilon) \times \H, F)$, define the $L^2$-inner products on $\H_t$ by
\[
	\langle f_1, f_2\rangle_{L^2}(t) := \langle f_1(t, \cdot), f_2(t, \cdot) \rangle_{L^2} := \int_{\H_t} a(f_1(t, \cdot), f_2(t, \cdot))d\mu_\sigma,
\]
where $d\mu_\sigma$ is the volume element associated to $\sigma$. 
The $L^2$-norm is defined as $\norm{f}_{L^2}(t) := \sqrt{\langle f, f\rangle_{L^2}(t)}$. 
For any $s \in \R$, we define the Sobolev inner products
\[
	\langle f_1, f_2 \rangle_{s}(t) := \langle (1 +\Delta)^{s/2}f_1, (1 +\Delta)^{s/2}f_2 \rangle_{L^2}(t)
\]
and the Sobolev norms 
\begin{equation} \label{eq: Sobolev_sigma}
	\norm{f}_{s}(t) := \sqrt{\langle f, f \rangle_{s}(t)}. 
\end{equation}

\subsubsection{Some estimates}

Let us prove two important but straightforward lemmas. 
We say that a linear differential operator $D$ \emph{differentiates in $\H_t$-direction} if for all $u \in C^\infty([0, \e) \times \H, F)$, $Du|_{\H_t}$ only depends on $u|_{\H_t}$ and \emph{not} on the $\d_t$-derivative of $u$. 
For example, $\Delta$ differentiates only in $\H_t$-direction.

\begin{lemma} \label{le: commutators1}
Let $m \in \N$ and let $D$ be a differential operator acting on sections of $F$ and differentiating in $\H_t$-direction of order $k \in \N$. Then the differential operators $[(1+\Delta)^m, D]$ and $[(1+\Delta)^m, \n_t]$ are differentiating in $\H_t$-direction and are of order $2m+k-1$ and $2m$ respectively.
\end{lemma}
\begin{proof}
Since the principal symbol of $\Delta$ is given by $-\sigma(\xi, \xi) \id_F$, it commutes with any other principal symbol. 
Hence the principal symbol of $(1+\Delta)^m$ commutes with the principal symbol of $D$ which proves that the $2m+k$-principal symbol of $[(1+\Delta)^m, D]$ vanishes. 
We conclude that $[(1+\Delta)^m, D]$ is of order $2m+k-1$. 

For a smooth vector field $X$, $[\n_t, \n_X] = R^\n(\d_t, X) + \n_{[\d_t, X]}$, where $R^\n$ is the curvature on $F$ induced by $\n$. 
It follows that $[\n_t, \n_X]$ is a differential operator of first order. 
If $X \in T\H_t$ for all $t \in [0, \e)$, then $[\d_t, X] \in T\H_t$ as well. 
By repeating this argument, we conclude that $[(1+\Delta)^m, \n_t]$ is a differential operator of order $2m$ only differentiating in $\H_t$-direction. 
\end{proof}

\begin{remark} \label{rmk: psi_over_t}
On many places in the proof of the energy estimates we will use the following trivial fact. 
Let $f \in C^\infty([0, \e) \times \H, \R)$ such that $f|_{t = 0}$. 
By compactness of $\H$, there is is a constant $C_f > 0$ such that
\[
	\max_{x \in \H}(f(t, \cdot)) \leq C_f t 
\]
for all $t>0$ small enough. 
If in addition $f \in C^\infty([0, \e) \times \H, \R)$ and $\d_tf|_{t = 0} > 0$, then there is a constant $C_f > 0$ such that
\begin{equation} \label{eq: estimate_C_f}
	\frac{t}{C_f} \leq \max_{x \in \H}(f(t, \cdot)) \leq C_f t.
\end{equation}
for all $t>0$ small enough. 
Proposition \ref{prop: time_function} implies that $\psi$ satisfies an inequality like \eqref{eq: estimate_C_f}.
\end{remark}

\begin{lemma} \label{le: commutators2}
Let $m \in \N$ and let $D$ be a linear differential operator on $F$ of order $k \in \N$, differentiating in $\H_t$-direction. 
For any $\a \in \R$ and smooth vector field $X$ such that $X \in T\H_t$ for all $t \in [0, \e)$, the operators $[D, \psi^\alpha]$, $[D, \d_X(\psi^\a)]$ and $[D, \d_t \psi]$ are differential operators of order $k-1$.
Moreover, for any $s \in \R$, there is a constant $C_s > 0$ such that for all $[t_0, t_1] \subset [0, \e)$,
\begin{align*}
	\norm{[D, \psi^\alpha] u(t, \cdot)}_{s} &\leq Ct^{\a + 1}\norm{u(t, \cdot)}_{s+k-1}, \\
	\norm{[D, \d_X(\psi^\a)] u(t, \cdot)}_{s} &\leq Ct^{\a + 1}\norm{u(t, \cdot)}_{s+k-1}, \\
	\norm{[D, (\d_t\psi) \psi^\a] u(t, \cdot)}_{s} &\leq Ct^{\a+1}\norm{u(t, \cdot)}_{s+k-1}.
\end{align*}
for any $s \in \R$ and $u \in C^\infty([t_0, t_1] \times \H, F)$ and $t \in [t_0, t_1]$.
\end{lemma}

\begin{proof}
Note that it suffices to prove the estimates for $u \in C^\infty([0, \e) \times \H, F)$.
Since $\psi^\a$, $\d_X (\psi^\a)$ and $\d_t \psi$ are scalar valued functions, it is clear that $[D, \psi^\alpha]$, $[D, \d_X(\psi^\a)]$ and $[D, (\d_t\psi) \psi^\a]$ are of order $k-1$. 
The key observation is that
\[
	\frac{\psi(t, \cdot)}{t} =  \frac{\psi(t, \cdot) - \psi(0, \cdot)}t \to \d_t\psi(0, \cdot) = 2
\]
as $t \to 0$ in $C^\infty(\H)$, since $\psi(0, \cdot) = 0$.
As a consequence, we get
\begin{align*}
	\frac{[D, \psi^\alpha]}{t^\a}u(t, \cdot) = [D,\left(\frac{\psi}t\right)^\alpha]u(t, \cdot) \to [D,2^\a]u(0, \cdot) = 0
\end{align*}
as $t \to 0$.
Therefore
\begin{align*}
	(1+\Delta)^{s/2}\frac{[D, \psi^\alpha]}{t^\a}u(t, \cdot)|_{t = 0} &= 0.
\end{align*}
We conclude that 
\begin{align*}
	a(1+\Delta)^{s/2}\frac{[D, \psi^\alpha]}{t^\a}u(t, \cdot), (1+\Delta)^{s/2}\frac{[D, \psi^\alpha]}{t^\a}u(t, \cdot))|_{t = 0} 
		&= 0, \\
	\d_t a(1+\Delta)^{s/2}\frac{[D, \psi^\alpha]}{t^\a}u(t, \cdot), (1+\Delta)^{s/2}\frac{[D, \psi^\alpha]}{t^\a}u(t, \cdot))|_{t = 0} 
		&= 0.
\end{align*}
Thus Remark \ref{rmk: psi_over_t} implies the first estimate.
The second and the third estimates follow along the same lines using $\d_t \psi(0, \cdot) = 2$.
\end{proof}

\subsubsection{The energy}

The proof of Theorem \ref{thm: Energy1} will rely on differentiating an energy with respect to the null time function. 
It turns out that the expression $\norm{u(t, \cdot)}_{2m+1} + \sqrt{t}\norm{\n_t u(t, \cdot)}_{2m}$ is not best suited for the calculations. 
Instead, we define another energy $\E^{2m}(u,t)$ which is equivalent to this and better suited for proving the energy estimate.
For this, we first define the expression
\[
	\g \circ a(\bar \n f_1, \bar \n f_2) := \sum_{i,j = 2}^n \g^{ij} a(\n_{e_i} f_1, \n_{e_j} f_2),
\]
for smooth sections $f_1, f_2$, where $(e_2, \hdots, e_n)$ is a frame in $E$. Recall that $\g$ is a positive definite metric on $E \subset T\H_t$ for all $t \in [0, \e)$.
Since $E \subset T\H_t$ is a subbundle, this expression is defined independently of the choice of frame.
We will also use the following natural notation
\[
	\langle \phi \bar \n f_1, \bar \n f_2\rangle_0(t, \cdot) := \int_{\H} \phi \g \otimes a(\bar \n f_1(t, \cdot), \bar \n f_2(t, \cdot)) d\mu_\sigma
\]
and
\[
	\norm{\phi \bar \n f}_0(t) := \sqrt{\langle \phi \bar \n f, \phi \bar \n f\rangle_0(t)},
\]
for smooth scalar valued functions $\phi$.
%Note that $\bar \n_X f|_\H = \n_X f|_\H$ for all $X \in E$.
With these defininitions at hand, we may define the energy.
For any $m \in \N$, we define the $2m$-energy as
\begin{align*}
	\E^{2m}(u, t) &:=  \norm{2\n_{\grad(t)} u + \psi \n_t u}_{2m}^2 + \norm{\sqrt{\psi} \bar \n (1+\Delta)^m u}_{0}^2 \\*
	& \quad + \norm{\sqrt \psi \n_tu}_{2m}^2 + \norm{\bar \n (1+\Delta)^m u}_{0}^2 \\*
	& \quad + \norm{u}_{2m}^2.
\end{align*}
Here and from now on $\norm{\cdot}_{s}$ means the Sobolev norm with respect to $\sigma$ as defined in equation \eqref{eq: Sobolev_sigma}.
The following lemma clarifies the relation between the Sobolev norms and the energy.

\begin{lemma} \label{le: Sobolev_vs_energy}
For each $m \in \N$ there is a constant $C_m > 0$ such that for all $[t_0, t_1] \subset [0, \e)$, we have
\begin{equation} \label{eq: Sobolev_vs_energy}
	\frac1{C_m} \sqrt {\E^{2m}(u, t)} \leq \norm{u(t, \cdot)}_{2m+1} + \sqrt t \norm{\n_t u(t, \cdot)}_{2m}  \leq C_m \sqrt {\E^{2m}(u, t)}
\end{equation}
for all $u \in C^\infty([t_0, t_1] \times \H, F)$ and $t \in [t_0, t_1]$. 
\end{lemma}

\begin{proof}
Recall from Section \ref{sec: null time function} that $\grad(t) = Z - \psi \d_t$, where $Z|_\H = -V$ and $Z \in T\H_t$ for all $t \in [0, \e)$. 
Moreover, we know that $T\H_t = \R Z \oplus E|_{t}$ for all $t \in [0, \e)$. 
Since $\H$ is compact and $\bar g$ is a Riemannian metric for all $t \in [0, \e)$, the norm $\norm{\n_Z u}_{2m} + \norm{\bar \n (1+\Delta)^mu}_0 + \norm{u}_{2m}$ is equivalent to $\norm{u}_{2m+1}$ for all $t \in [0, \e)$. 
Let $C$ denote some constant depending on $m$ but not on $u$. 
We have
\begin{align*}
	\norm{u}_{2m+1} &\leq C (\norm{\n_Z u}_{2m} + \norm{\bar \n (1+\Delta)^m u}_0 + \norm{u}_{2m}) \\
		&\leq C\norm{2\n_Zu - \psi \n_t u}_{2m} + C \norm{\psi\n_t u}_{2m} + C\sqrt{\E^{2m}(u,t)} \\
		&\leq C\norm{2\n_{\grad(t)}u + \psi \n_t u}_{2m} + C \norm{[(1+\Delta)^m, \sqrt \psi] \sqrt \psi (\n_t u)}_{0} \\
		&\quad + C \norm{\sqrt \psi(1+\Delta)^m(\sqrt \psi \n_t u)}_{0}+ C\sqrt{\E^{2m}(u,t)} \\
		&\leq C \sqrt{\E^{2m}(u,t)},
\end{align*}
by Lemma \ref{le: commutators2}.
We also have
\begin{align*}
	\norm{\n_t u}_{2m}
		& \leq C \norm{[(1+\Delta)^m, \frac1{\sqrt \psi}](\sqrt{\psi}\n_t u)}_{0} + C \norm{ \frac1{\sqrt \psi}(1+\Delta)^m(\sqrt{\psi}\n_t u)}_{0} \\
		& \leq \frac C {\sqrt t} \sqrt{\E^{2m}(u,t)},
\end{align*}
by Remark \ref{rmk: psi_over_t} and Lemma \ref{le: commutators2}.
The other direction is similar.
\end{proof}

\subsection{Proof of the energy estimate} \label{sec: Energy_proof}

We first derive the main part of the energy estimate for a specific wave operator defined as
\[
	Q:= - \n_t (\psi \n_t + 2\n_{\grad(t)}) + \bar \Delta,
\]
where  
\[
	\bar \Delta f := - \sum_{i = 2}^n \g^{ij}\left(\n_{e_i} \n_{e_j} - \n_{\n_{e_i}e_j} \right)f,
\]
for some frame $(e_i)_{i=2}^n$ of $E$.
Again, since $E \subset T\H_t$ is a subbundle, the operator $\bar \Delta$ is independent of the choice of frame. 
Let us for completeness check that $Q$ is indeed a wave operator.
By Proposition \ref{prop: time_function}, the metric is given by
\[
	g_{\a \b} = \begin{pmatrix}
	0 & 1 & 0 \\
	1 & -\psi & 0 \\
	0 & 0 & \g_{ij}
	\end{pmatrix}
	\Rightarrow
	g^{\a \b} = \begin{pmatrix}
	\psi & 1 & 0 \\
	1 & 0 & 0 \\
	0 & 0 & \g^{ij}
	\end{pmatrix},
\]
for $i,j \geq 2$, in the basis $(\d_t, \grad(t), e_2, \hdots, e_n)$, where $e_2, \hdots, e_n \in E$.
Since the leading order term of $Q$ is of the form $-g^{\a\b}\d_{e_\a} \d_{e_\b}$ with respect to this frame, we conclude that $Q$ is a wave operator.

In this section, $\norm{\cdot}_s$ denotes the Sobolev norm on $H^{s}(\H)$ with respect to $\sigma$ as defined in equation \eqref{eq: Sobolev_sigma}.
The following lemma is the essential estimate in the proof of Theorem \ref{thm: Energy1}.
Recall from Section \ref{sec: null time function} that $\grad(t) = Z - \psi \d_t$, where $Z|_\H = -V$ and $Z \in T\H_t$ for all $t \in [0, \e)$. 

\begin{lemma}[The energy estimate for the operator $Q$] \label{lemma: Q-estimate}
For any integer $m \in \N$ there is a constant $C > 0$ such that for all $[t_0, t_1] \subset (0, \e)$, we have
\begin{align*}
		\frac{d}{dt} \E^{2m}(u,t)& \leq \frac{C}t\E^{2m}(u,t) - 2 \Re \langle{Q u, 2\n_Z u - \psi \n_t u } \rangle_{2m} \\
			&\quad + 2 \Re \langle \frac{1}{\sqrt \psi} Q u, \sqrt \psi \n_t u\rangle_{2m}. \label{eq: first_estimate}
\end{align*}
for any $u \in C^\infty([t_0, t_1] \times \H, F)$ and each $t \in [t_0, t_1]$.
\end{lemma}

\noindent An important tool in the proof of Lemma \ref{lemma: Q-estimate} is the following \enquote{integration by parts} estimate.

\begin{lemma} \label{le: part_integration}
There exists a $C > 0$ such that for all $[t_0, t_1] \subset [0, \e)$ we have 
\begin{equation}
	\abs{\langle \bar \Delta f_1, f_2 \rangle_{0} - \langle \bar \n f_1, \bar \n f_2 \rangle_{0}} \leq C \norm{f_1}_{1}\norm{f_2}_{0} + Ct \norm{\n_tf_1}_{0} \norm{f_2}_{0}, \label{eq: part_integration}
\end{equation}
for all $f_1, f_2 \in C^\infty([t_0, t_1] \times \H, F)$ and for each $t \in [0,\e)$.
\end{lemma}
\begin{proof}[Proof of Lemma \ref{le: part_integration}]
The idea is to compute the divergence with respect to $\sigma$ of a certain one-form and apply the Stokes theorem.
For each smooth vector field $X \in T\H_t$, $\sigma(X, \cdot)$ is a smooth one-form on $\H_t$.
We restrict the two one-forms $\sigma(X, \cdot)$ and $a(\n_{\cdot}f_1, f_2)$ to $E \subset T\H_t$ and use the positive definite metric $\g$ on $E$ to define the one-form $\Omega$ on $T\H_t$ by
\[
	\Omega(X) := \g(\sigma(X, \cdot), a(\n_{\cdot}f_1, f_2)) = \sum_{i,j = 2}^n \sigma(X, e_i)\g^{ij}a(\n_{e_j} f_1, f_2),
\]
where $e_2, \hdots, e_n$ is a frame of $E$.
We want to compute the divergence of $\Omega$ on $T\H_t$ with respect to $\sigma$.
The computation relies on the crucial feature that
\[
	\Omega(V) = 0,
\]
since $\sigma(V, e_i) = 0$ for $i = 2, \hdots, n$.
Choose a $\sigma$-orthonormal frame $(e_1 = V, e_2, \hdots, e_n)$ with $e_2, \hdots, e_n \in E$. 
We calculate
\begin{align*}
	\div_\sigma(\Omega)
		&= \sum_{k = 1}^n \d_{e_k} \Omega(e_k) - \Omega(\hat \n_{e_k}e_k) \\
		&= \sum_{k = 2}^n \d_{e_k} \Omega(e_k) - \sum_{k = 1}^n\Omega(\hat \n_{e_k}e_k) \\
		&= \sum_{j,k = 2}^n \left(\d_{e_k}(\g^{kj})a(\n_{e_j}f_1, f_2) + \g^{kj}\n_{e_k}a(\n_{e_j}f_1, f_2)\right) \\
			& \quad + \sum_{j,k = 2}^n \left(\g^{kj}a(\n_{e_k} \n_{e_j}f_1, f_2) + \g^{kj}a(\n_{e_j}f_1, \n_{e_k}f_2) \right) \\
			& \quad - \sum_{\substack{i,j = 2 \\ k = 1}}^n \sigma(\hat \n_{e_k}e_k, e_i)\g^{ij}a(\n_{e_j}f_1, f_2) \\
		&= - a(\bar \Delta f_1, f_2) + \g \otimes a(\bar \n f_1, \bar \n f_2) \\
		& \quad + \sum_{j,k = 2}^n \left(\d_{e_k}(\g^{kj})a(\n_{e_j}f_1, f_2) + \g^{kj}(\n_{e_k}a)(\n_{e_j}f_1, f_2)\right) \\
		& \quad + \sum_{j,k = 2}^n\g^{kj}a(\n_{\n_{e_k}e_j}f_1, f_2) - \sum_{\substack{i,j = 2 \\ k = 1}}^n \sigma(\hat \n_{e_k}e_k, e_i)\g^{ij}a(\n_{e_j}f_1, f_2).
\end{align*}
Since $\sigma$ is a metric on $\H_t$ it follows that $\hat \n_{e_i}e_i \in T\H_t$ and since $\H$ is totally geodesic, it follows that $\n_{e_i}e_i|_\H \in T\H$.
In other words, all coefficients in front of the term $a(\n_tf_1, f_2)$ vanish at $t = 0$.
Integrating over $\H_t$ using Stokes' theorem and applying Remark \ref{rmk: psi_over_t} implies the assertion.
\end{proof}

\begin{proof}[Proof of Lemma \ref{lemma: Q-estimate}]
We will throughout the proof use Lemma \ref{le: commutators1}, Remark \ref{rmk: psi_over_t}, Lemma \ref{le: commutators2} and Lemma \ref{le: Sobolev_vs_energy} without explicitly mentioning it. 
The symbol $C$ will in this proof denote a constant depending only on $m$, $Q$ and the geometry, its value can change from line to line. 
Let us write $\grad(t) = -\psi \d_t + Z$, where $Z \in T\H_t$ for all $t \in  [0, \e)$. 
The operator $Q$ takes the form 
\[
	Q = \n_t (\psi \n_t - 2\n_Z) + \bar \Delta.
\]
We start by differentiating the first term of the energy.
Inserting the definition of $Q$, we get
\begin{align}
\frac{d}{dt} &\norm{\psi \n_t u - 2\n_Z u}_{2m}^2 \nonumber \\
	&= \frac{d}{dt} \int_\H \abs{(1+\Delta)^m (\psi \n_t u - 2\n_Z u)}^2_a d\mu_\sigma \nonumber \\
	&= \int_\H (\n_ta)((1+\Delta)^m(\psi \n_t u - 2\n_Z u)), (1+\Delta)^m(\psi \n_t u - 2\n_Z u))) d\mu_\sigma \nonumber \\*
	& \quad + 2 \Re \langle [\n_t, (1+\Delta)^m](\psi \n_t u - 2\n_Z u ), (1+\Delta)^m(\psi \n_t u  - 2\n_Z u) \rangle_{0} \nonumber \\*
	& \quad + 2 \Re \langle \n_t(\psi \n_t u - 2\n_Z u ), \psi \n_t u  - 2\n_Z u \rangle_{2m} \nonumber \\
	&\leq C\E^{2m}(u, t) + 2 \Re \langle{ Q u - \bar \Delta u, \psi \n_t u - 2\n_Z u} \rangle_{2m} \nonumber \\
	&\leq C\E^{2m}(u, t) + 2 \Re \langle{ Q u, \psi \n_t u - 2\n_Z u} \rangle_{2m} - 2 \Re \langle \bar \Delta u , \psi \n_t u - 2 \n_Z u \rangle_{2m}. \label{eq: higher_estimate_second}
\end{align}
The last term in equation \eqref{eq: higher_estimate_second} is estimated using Lemma \ref{le: part_integration} as follows:
\begin{align}
	- 2 \Re \langle \bar \Delta u , \psi \n_t u - 2 \n_Z u\rangle_{2m} 
		&= -2\Re \langle [(1+\Delta)^m, \bar \Delta] u, (1+\Delta)^m(\psi \n_t u - 2\n_Z u) \rangle_{0} \nonumber \\*
			&\quad - 2 \Re \langle \bar \Delta (1+\Delta)^m u, (1+\Delta)^m(\psi \n_t u - 2 \n_Z u) \rangle_{0} \nonumber \\
		& \leq C \E^{2m}(u,t) \nonumber \\*
			&\quad - 2\Re \langle \bar \n (1+\Delta)^m u, \bar \n (1+\Delta)^m(\psi \n_t u - 2 \n_Z u )\rangle_{0} \nonumber \\
		& = C\E^{2m}(u, t) \nonumber \\*
			& \quad - 2\Re \langle \bar \n (1+\Delta)^m u, \bar \n (1+\Delta)^m (\psi  \n_tu)\rangle_0 \nonumber \\*
			& \quad + 4 \Re \langle \bar \n (1+\Delta)^m u, \bar \n (1+\Delta)^m \n_Z u \rangle_{0} \nonumber \\
		& \leq C \E^{2m}(u, t) \nonumber \\*
			& \quad - 2\Re \langle \bar \n (1+\Delta)^m u, [\bar \n (1+\Delta)^m, \psi]  \n_tu)\rangle_{0} \nonumber \\*
			& \quad - 2\Re \langle \bar \n (1+\Delta)^m u, \psi \bar \n [(1+\Delta)^m, \n_t]u\rangle_{0} \nonumber \\*
			& \quad - 2\Re \langle \bar \n (1+\Delta)^m u,  \psi \bar \n \n_t(1+\Delta)^m u\rangle_{0} \nonumber \\*
			& \quad + 4\Re \langle \bar \n (1+\Delta)^m u, \bar \n [(1+\Delta)^m, \n_Z] u \rangle_{0} \nonumber \\*
			& \quad + 4\Re \langle \bar \n (1+\Delta)^m u, \bar \n \n_Z (1+\Delta)^m u \rangle_{0} \nonumber \\
		& \leq C \E^{2m}(u,t) \nonumber \\*
			& \quad - 2\Re\langle \bar \n (1+\Delta)^m u,  \psi \bar \n \n_t(1+\Delta)^m u\rangle_{0} \nonumber \\*
			& \quad + 4\Re \langle \bar \n (1+\Delta)^m u, \bar \n \n_Z (1+\Delta)^m u \rangle_{0}. \label{eq: ee_second_first}
\end{align}
In order to estimate the remaining terms in equation \eqref{eq: ee_second_first}, we need the following observation.
For any smooth vector field $X$ such that $X \in T\H_t$ for all $t \in [0,\e)$, we have $[Z, X], [\d_t, X] \in T\H_t$ for all $t \in [0, \e)$. 
It follows that $[\n_Z, \n_X] = R^\n(Z, X) + \n_{[Z, X]}$ and $[\n_t, \n_X] = R^\n(\d_t, X) + \n_{[\d_t, X]}$ are first order differential operators only differentiating in $\H_t$-direction. 
Using this, it is clear that
\[
	\d_Z \left( \g \otimes a (\bar \n v, \bar \n v) \right) - 2\Re \left( \g \otimes a (\bar \n v, \bar \n \n_Z v) \right)
\]
and 
\[
	\d_t \left( \g \otimes a (\bar \n v, \bar \n v) \right) - 2 \Re \left(\g \otimes a (\bar \n v, \bar \n \n_t v) \right)
\]
only depend on first order spatial derivatives in $v$ and we conclude that
\begin{align}
	\abs{ \int_\H \d_Z \left( \g \otimes a (\bar \n v, \bar \n v) \right) - 2\Re \left( \g \otimes a (\bar \n v, \bar \n \n_Z v) \right)  d\mu_\sigma} 
		&\leq C \norm{v}^2_{1}, \nonumber \\*
	\abs{\int_\H \d_t \left( \g \otimes a (\bar \n v, \bar \n v) \right) - 2 \Re \left( \g \otimes a (\bar \n v, \bar \n \n_t v) \right) d\mu_\sigma} 
		&\leq C \norm{v}^2_{1}. \label{eq: t-estimate}
\end{align}
Using this with $v = (1+\Delta)^m u$ and Stokes' theorem, we can continue the estimate \eqref{eq: ee_second_first} as
\begin{align}
	- 2 \Re \langle \bar \Delta u , \psi \n_t u - 2 \n_Z u\rangle_{2m} 
		& \leq C \E^{2m}(u,t) - \frac{d}{dt} \norm{\sqrt \psi \bar \n (1 + \Delta)^m u}_0^2 \nonumber \\*
			& \quad + \int_{\H} (\d_t\psi) \abs{\bar \n(1 + \Delta)^m u}_{\g \otimes a}^2 d\mu_\sigma \nonumber \\*
			& \quad + C \int_\H \d_Z \abs{\bar \n (1+\Delta)^m u}_{\g \otimes a}^2d\mu_\sigma \nonumber \\ 
		& \leq C \E^{2m}(u,t) - \frac{d}{dt} \norm{\sqrt \psi \bar \n (1 + \Delta)^m u}_0^2 \nonumber \\*
			& \quad - C \int_{\H} \div_\sigma(Z)\abs{\bar \n (1+\Delta)^m u}_{\g \otimes a}^2d\mu_\sigma \nonumber \\
		& \leq C \E^{2m}(u,t) - \frac{d}{dt} \norm{\sqrt \psi \bar \n (1 + \Delta)^m u}_0^2. \label{eq: ee_second_first_extra}
\end{align}
Combining estimates (\ref{eq: higher_estimate_second} - \ref{eq: ee_second_first_extra}) gives
\begin{align}
	\frac{d}{dt} &\left( \norm{2\n_Z u - \psi \n_t u}_{2m}^2 + \norm{\sqrt \psi \bar \n (1+\Delta)^m u}_{0}^2 \right) \nonumber \\
	&\leq C\E^{2m}(u,t) - 2 \Re \langle{Q u, 2\n_Z u - \psi \n_t u } \rangle_{2m}. \label{eq: first_estimate}
\end{align}
Let us continue with the second term in the energy, we have
\begin{align}
	\frac{d}{dt}\left( \norm{\sqrt \psi \n_t u}_{2m}^2 \right) 
	&= \int_\H (\n_ta)((1+\Delta)^m(\sqrt \psi \n_t u), (1+\Delta)^m(\sqrt \psi \n_t u) )d\mu_\sigma \nonumber \\
		& \quad + 2 \Re \langle [\n_t, (1+\Delta)^m](\sqrt \psi \n_t u), (1+\Delta)^m(\sqrt \psi \n_t u) \rangle_0 \nonumber \\
		& \quad + 2 \Re \langle \n_t \left( \sqrt \psi \n_t u \right), \sqrt \psi \n_t u \rangle_{2m} \nonumber \\
	&\leq C \E^{2m}(u, t) + 2 \Re \langle \frac1{\sqrt \psi} \n_t(\psi\n_t u), \sqrt \psi \n_t u \rangle_{2m} \nonumber \\*
	& \quad - \Re \langle \frac{\d_t \psi}{\psi} \sqrt \psi \n_t u, \sqrt \psi \n_t u \rangle_{2m} \nonumber \\
	&= C \E^{2m}(u, t) + 2 \Re \langle \frac{1}{\sqrt \psi} Q u, \sqrt \psi \n_t u\rangle_{2m} + 4 \Re \langle \frac{1}{\sqrt \psi} \n_Z \n_t u, \sqrt \psi \n_t u\rangle_{2m} \nonumber \\*
	& \quad + 4 \Re \langle \frac{1}{\sqrt \psi}\n_{[\d_t, Z]} u, \sqrt \psi \n_t u\rangle_{2m} - 2\Re \langle \frac{1}{\sqrt \psi} \bar \Delta u, \sqrt \psi \n_t u\rangle_{2m} \nonumber \\*
	& \quad - \Re \langle \frac{\d_t \psi}{\psi} \sqrt \psi \n_t u, \sqrt \psi \n_t u \rangle_{2m} . \label{eq: first_higher_estimate}
\end{align}
We estimate the third term of equation \eqref{eq: first_higher_estimate} as
\begin{align}
	 4 \Re \langle \frac{1}{\sqrt \psi} &\n_Z \n_t u, \sqrt \psi \n_t u\rangle_{2m} \nonumber \\*
	&= 4 \Re \langle \frac 1 \psi \n_Z(\sqrt \psi \n_t u), \sqrt \psi \n_t u \rangle_{2m} + 2 \Re \langle \d_Z(\frac1 \psi) \sqrt \psi \n_t u, \sqrt \psi \n_t u \rangle_{2m} \nonumber \\
	& = 4 \Re \langle [(1+ \Delta)^m, \frac{1}{\psi}]\n_Z(\sqrt \psi \n_t u), (1+\Delta)^m(\sqrt \psi \n_t u) \rangle_{0} \nonumber \\*
	& \quad + 4 \Re \langle \frac1 \psi [(1 + \Delta)^m, \n_Z](\sqrt{\psi}\n_tu), (1+\Delta)^m(\sqrt \psi \n_t u)\rangle_{0} \nonumber \\*
	& \quad + 2 \int_\H \frac1 \psi \d_Z \abs{(1 + \Delta)^m(\sqrt \psi \n_t u)}^2_{a} d\mu_\sigma \nonumber \\*
	& \quad - 2 \int_\H \frac1 \psi (\n_Za)((1 + \Delta)^m(\sqrt \psi \n_t u), (1 + \Delta)^m(\sqrt \psi \n_t u)) d\mu_\sigma \nonumber \\*
	& \quad + 2 \Re \langle [(1 + \Delta)^m, \d_Z(\frac{1}{\psi})](\sqrt \psi \n_tu), (\sqrt \psi \n_tu) \rangle_{0} \nonumber \\
	& \quad + 2 \int_\H \d_Z(\frac1 \psi) \abs{(1 + \Delta)^m(\sqrt \psi \n_t u)}_a^2 d\mu_\sigma \nonumber \\*
	& \leq \frac{C}{t}\E^{2m}(u,t) + 2 \int_\H \d_Z\left( \frac1 \psi \abs{(1 + \Delta)^m(\sqrt \psi \n_t u)}_a^2 \right) d\mu_\sigma \nonumber \\
	& \leq \frac{C}{t}\E^{2m}(u,t) - 2 \int_{\H} \frac{\div_\sigma(Z)}{\psi} \abs{(1 + \Delta)^m(\sqrt \psi \n_t u)}^2_{a} d\mu_\sigma \nonumber \\*
	& \leq \frac{C}{t}\E^{2m}(u,t).
\end{align}
The fourth term in equation \eqref{eq: first_higher_estimate} is estimated as
\begin{align}
	2 \Re \langle \frac{1}{\sqrt \psi} 2\n_{[\d_t, Z]} u, \sqrt \psi \n_t u\rangle_{2m} &\leq C\norm{[(1+\Delta)^m, \frac{1}{\sqrt \psi}] \n_{[\d_t, Z]}u}_{0} \sqrt{\E^{2m}(u,t)} \nonumber \\*
		& \quad + C\norm{\frac{1}{\sqrt \psi} (1+\Delta)^m \n_{[\d_t, Z]}u}_{0} \sqrt{\E^{2m}(u,t)} \nonumber \\
		&\leq \frac{C}{\sqrt{t}}\E^{2m}(u,t) \label{eq: ee_first_fourth}
\end{align}
since $[\d_t, Z] \in T\H_t$ for all $t \in [0,\epsilon)$. The fifth term in equation \eqref{eq: first_higher_estimate} is estimated using Lemma \ref{le: part_integration} as
\begin{align}
	-2 \Re \langle \frac{1}{\sqrt \psi} \bar \Delta u, \sqrt \psi \n_t u\rangle_{2m} 
		&= -2 \Re \langle [(1+\Delta)^m, \frac{1}{\sqrt \psi}] \bar \Delta u, (1+\Delta)^m(\sqrt \psi \n_t u) \rangle_{0} \nonumber \\*
		& \quad - 2 \Re \langle \frac{1}{\sqrt \psi} [(1+\Delta)^m, \bar \Delta] u, (1+\Delta)^m(\sqrt \psi \n_t u) \rangle_{0} \nonumber \\*
		& \quad - 2 \Re \langle \bar \Delta (1+\Delta)^m u, \frac{1}{\sqrt \psi} (1+\Delta)^m(\sqrt \psi \n_t u) \rangle_{0} \nonumber \\
	& \leq \frac{C}{\sqrt t} \E^{2m}(u,t) \nonumber \\*
		& \quad - 2 \Re \langle \bar \n (1+\Delta)^m u, \bar \n(\frac{1}{\sqrt \psi} (1+\Delta)^m(\sqrt \psi \n_t u))\rangle_{0} \nonumber \\
	& = \frac{C}{\sqrt t} \E^{2m}(u,t) \nonumber \\*
		& \quad - 2 \Re \langle \bar \n (1+\Delta)^m u, [\bar \n, \frac{1}{\sqrt \psi}] (1+\Delta)^m(\sqrt \psi \n_t u)\rangle_{0} \nonumber \\*
		& \quad - 2 \Re \langle \bar \n (1+\Delta)^m u, \frac{1}{\sqrt\psi}\bar \n (1+\Delta)^m(\sqrt \psi \n_t u)\rangle_{0} \nonumber \\
	& \leq \frac{C}{\sqrt t} \E^{2m}(u,t) \nonumber \\*
		& \quad - 2 \Re \langle \bar \n (1+\Delta)^m u, \frac{1}{\sqrt\psi}[\bar \n (1+\Delta)^m, \sqrt \psi] \n_t u\rangle_{0} \nonumber \\
		& \quad - 2 \Re \langle \bar \n (1+\Delta)^m u, \bar \n [(1+\Delta)^m, \n_t] u\rangle_{0} \nonumber \\
		& \quad - 2 \Re \langle \bar \n (1+\Delta)^m u, \bar \n \n_t (1+\Delta)^m u\rangle_{0} \nonumber \\
	& \leq \frac{C}{\sqrt t} \E^{2m}(u,t) - \frac{d}{dt}\norm{\bar \n(1+ \Delta)^m u}_0^2, \label{eq: ee_first_fifth}
\end{align}
where we in the last line used equation \eqref{eq: t-estimate}.
Using that $\frac{\d_t \psi}{\psi} > 0$, the last term in equation \eqref{eq: first_higher_estimate} is estimated as
\begin{align}
	- \Re \langle \frac{\d_t \psi}{\psi} \sqrt \psi \n_t u, \sqrt \psi \n_t u \rangle_{2m} 
		&= - \Re \langle (1+\Delta)^m (\frac{\d_t\psi}{\psi} \sqrt \psi \n_t u), (1+\Delta)^m(\sqrt \psi \n_t u) \rangle_{0} \nonumber \\
		&= - \Re \langle [(1+\Delta)^m, \frac{\d_t \psi}{\psi}] (\sqrt \psi \n_t u), (1+\Delta)^m(\sqrt \psi \n_t u) \rangle_{0} \nonumber \\*
			&\quad - \int_{\H}\frac{\d_t \psi}{\psi} \abs{(1+\Delta)^m (\sqrt \psi \n_t u)}_a^2 d\mu_\sigma \nonumber \\
		&\leq C\E^{2m}(u,t). \label{eq: ee_first_first}
\end{align}
Combine the estimates (\ref{eq: first_higher_estimate} - \ref{eq: ee_first_first}) to get
\begin{align}
	\frac{d}{dt}\left(\norm{\sqrt \psi \n_t u}_{2m}^2 + \norm{\bar \n (1+\Delta)^m u}_0^2 \right) 
		&\leq \frac{C}t \E^{2m}(u,t) + 2 \Re \langle \frac{1}{\sqrt \psi} Q u, \sqrt \psi \n_t u\rangle_{2m}. \label{eq: second_estimate}
\end{align}
The last term in the energy is estimated as
\begin{align}
	\frac{d}{dt} \left(\norm{u}_{2m}^2 \right) 
		&=\int_{\H} (\n_t a)((1+\Delta)^mu, (1+\Delta)^mu) d\mu_\sigma \nonumber \\*
			& \quad + \Re \langle [\n_t, (1+\Delta)^m]u, (1+\Delta)^mu \rangle_{0}+ 2 \Re \langle \n_t u, u \rangle_{2m} \nonumber \\
		&\leq C\E^{2m}(u,t) +  2 \Re \langle \frac{1}{\sqrt \psi}[\sqrt \psi, (1 + \Delta)^m]\n_t u , (1 +\Delta)^mu \rangle_{0} \nonumber \\*
			&\quad + 2 \Re \langle \frac{1}{\sqrt \psi}(1 + \Delta)^m(\sqrt \psi \n_t u) , (1 +\Delta)^m u \rangle_{0} \nonumber \\
		&\leq \frac{C}{\sqrt t}\E^{2m}(u,t). \label{eq: last_estimate}
\end{align}
Combining the estimates \eqref{eq: first_estimate}, \eqref{eq: second_estimate} and \eqref{eq: last_estimate} proves the assertion.
\end{proof}

To prove Theorem \ref{thm: Energy1} using Lemma \ref{lemma: Q-estimate} is now straightforward.

\begin{proof}[Proof of Theorem \ref{thm: Energy1}]
Since both $Q$ and $P$ are wave operators, $R := Q - P$ is a first order differential operator. Inserting $Q = P + R$ into Lemma \ref{lemma: Q-estimate} gives
\begin{align*}
		\frac{d}{dt} \E^{2m}(u,t)& \leq \frac{C}{t}\E^{2m}(u,t) - 2 \Re \langle{Pu+R u, 2\n_Z u - \psi \n_t u } \rangle_{2m} \\*
			&\quad + 2 \Re \langle \frac{1}{\sqrt \psi} (Pu+R u), \sqrt \psi \n_t u\rangle_{2m} \\*
			&\leq \frac{C}{t}\E^{2m}(u,t) + \frac{C}{\sqrt t}\norm{Pu + Ru}_{2m}\sqrt{\E^{2m}(u,t)}\\*
			&\quad + C \norm{[(1 + \Delta)^m, \frac{1}{\sqrt \psi}] (Pu+R u)}_0 \sqrt{\E^{2m}(u,t)} \\*
			&\leq \frac Ct \E^{2m}(u,t) + \frac{C}{\sqrt t}(\norm{Pu}_{2m} + \norm{Ru}_{2m})\sqrt{\E^{2m}(u,t)},
\end{align*}
by Lemma \ref{le: commutators2}.
Now, since $R$ is a differential operator of first order, we can estimate 
\begin{align*}
	\norm{Ru}_{2m} 
		& \leq C \norm{u}_{2m+1} + C\norm{\n_tu}_{2m} \\*
		& \leq \frac{C}{\sqrt t} \sqrt{\E^{2m}(u,t)},
\end{align*}
by Lemma \ref{le: Sobolev_vs_energy}.
Altogether, we have shown that 
\begin{align*}
	\frac{d}{dt} \E^{2m}(u,t)& \leq \frac Ct \E^{2m}(u,t) + \frac{C}{\sqrt t} \norm{Pu}_{2m} \sqrt{\E^{2m}(u,t)},
\end{align*}
which in turn implies that
\begin{align*}
	\frac{d}{dt} \sqrt{\E^{2m}(u,t)} & \leq \frac Ct \sqrt{\E^{2m}(u,t)} + \frac{C}{\sqrt t} \norm{Pu}_{2m},
\end{align*}
In other words,
\begin{align*}
	\frac{d}{dt}\left(\frac{\sqrt{\E^{2m}(u,t)}}{t^C} \right) \leq \frac{C}{t^{C+1/2}} \norm{Pu}_{2m}.
\end{align*}
Integrating this proves the statement.
\end{proof}

\section{Proof of existence and uniqueness given an asymptotic solution} \label{sec: global_given_asymptotic}

The purpose of this section is to prove Theorem \ref{mainthm1}.
The main ingredient in the proof is the energy estimate, Theorem \ref{thm: Energy1}.
The idea is to solve a sequence of Cauchy problems with initial data on Cauchy hypersurfaces approaching the Cauchy horizon. 
The initial data is expressed in terms of the asymptotic solution.
A careful use of the energy estimate shows that the sequence converges to an actual unique solution of the wave equation.
The solution will then be shown to coincide with the asymptotic solution up to any order at the Cauchy horizon, which also shows that it is smooth up to the Cauchy horizon.

Let us fix $m, N \in \N$ such that $N > D_m$, where $D_m > 0$ is the constant from Theorem \ref{thm: Energy1}. 
%We write $w := w^N$, keeping in mind that $w$ depends on $N$.
For each $\tau \in (0, \epsilon)$, define $v_\tau \in C^\infty((0, \epsilon)\times \H, F)$ by solving the Cauchy problem
\begin{align*}
	P v_\tau &= f, \\
	v_\tau(\tau, \cdot) &= w^N(\tau, \cdot), \\
	\n_t v_\tau(\tau, \cdot) &= \n_t w^N(\tau, \cdot).
\end{align*}
Since $(0, \epsilon) \times \H$ is globally hyperbolic, \cite{BaerGinouxPfaeffle2007}*{Thm. 3.2.11} guarantees the existence of a unique solution.

\begin{lemma}[Comparison of $v_\tau$ and $w^N$] \label{le: comparison_sequence}
There is a constant $C > 0$ (depending on $m$ and $N$) such that
\[
	\norm{(v_\tau - w^N)(t, \cdot)}_{2m+1} + \sqrt t \norm{\n_t(v_\tau - w^N)(t, \cdot)}_{2m} \leq Ct^{N + \frac32}
\]
for all $\tau, t \in (0, \e)$ such that $\tau \leq t$.
\end{lemma}
\begin{proof}

Note that at time $\tau$, we have
\[
	\norm{(v_\tau - w^N)(\tau, \cdot)}_{2m+1} + \sqrt \tau \norm{\n_t(v_\tau - w^N)(\tau, \cdot)}_{2m} = 0.
\]
Therefore Theorem \ref{thm: Energy1} with $t_0 = \tau$ and $t_1 = t$ implies that there is a $D_m > 0$ such that
\begin{align}
	\norm{(v_\tau - w^N)(t, \cdot)}_{2m+1} + \sqrt t \norm{\n_t(v_\tau - w^N)(t, \cdot)}_{2m} & \leq D_m t^{D_m}\int_\tau^t \frac{\norm{(f-Pw^N)(s, \cdot)}_{2m}}{s^{D_m + 1/2}}ds \nonumber \\
		& \leq D_m t^{D_m}\int_0^t \frac{\norm{(f-Pw^N)(s, \cdot)}_{2m}}{s^{D_m + 1/2}}ds. \label{eq: crucial integral}
\end{align}
Note that
\[
	\left(\frac{d}{dt}\right)^k \norm{(f - P w^N)(t, \cdot)}_{2m}^2 = \sum_{j=0}^k \genfrac(){0pt}{0}{k}{j} \langle (\n_t)^j (f - P w^N) (t, \cdot), (\n_t)^{k-j} (f - P w^N)(t, \cdot) \rangle_{2m}.
\]
By the assumptions in Theorem \ref{mainthm1}, we conclude that 
\[
	\left(\frac{d}{dt}\right)^k  \norm{(f - P w^N)(t, \cdot)}_{2m}^2|_{t=0} = 0
\]
for every $k \leq 2N+1$. 
Therefore there is a constant $C>0$ such that 
\begin{equation}
	\norm{(f - P w^N)(t, \cdot)}_{2m} \leq C t^{N+1} \label{eq: asympt}
\end{equation}
for all $t \in [0, \e)$. 
Since $N > D_m$, the integral in equation \eqref{eq: crucial integral} is bounded.
Therefore we may insert the bound \eqref{eq: asympt} into the integral and conclude the statement.
\end{proof}

The next step is to use the above lemma together with Rellich's Theorem to show that a subsequence $v_{\tau_j}$ converges to a limit $u$.

\begin{lemma}[The local existence]
There is a 
\[
	u \in C^0([0, \e), H^{2m}(\H)) \cap C^1([0, \e), H^{2m-1}(\H))
\]
such that
\[
	Pu = f
\]
and there is a constant $C > 0$ (depending on $m$ and $N$) such that
\[
	\norm{(u - w^N)(t, \cdot)}_{2m} + \sqrt t \norm{\n_t(u - w^N)(t, \cdot)}_{2m-1} \leq Ct^{N + \frac32}.
\]
\end{lemma}
\begin{proof}
Fix $T \in (0,\e)$. 
The previous lemma implies that $(v_{\tau}, \n_t v_{\tau})(T, \cdot)$ is bounded in $H^{2m+1}(\H) \times H^{2m}(\H)$. 
Therefore Rellich's Theorem implies that there is a sequence $\tau_j \to 0$ and $(a_0, a_1) \in H^{2m}(\H) \times H^{2m-1}(\H)$ such that
\begin{equation} \label{eq: conv_initial_data}
	(v_{\tau_j}, \n_t v_{\tau_j})(T, \cdot) \to (a_0, a_1) \in H^{2m}(\H) \times H^{2m-1}(\H),
\end{equation}
as $j \to \infty$. 
By Theorem \cite{BaerWafo2014}*{Thm. 13}, we may define $u \in C^0((0, \e), H^{2m}(\H)) \cap C^1((0, \e), H^{2m-1}(\H))$ by solving
\begin{align*}
	P u &= f, \\
	u(T, \cdot) &= a_0, \\
	\n_t u(T, \cdot) &= a_1.
\end{align*}
Strictly speaking \cite{BaerWafo2014}*{Thm. 13} does not apply, since the time function in \cite{BaerWafo2014} is not a null time function. 
However, it is straightforward to modify the proof of \cite{BaerWafo2014}*{Thm. 13} using our energy estimate Theorem \ref{thm: Energy1}, since we only consider $t > 0$.

Since 
\[
	\norm{(v_{\tau_j} - u)(T, \cdot)}_{2m} + \norm{\n_t(v_{\tau_j} - u)(T, \cdot)}_{2m-1} \to 0
\]
as $j \to \infty$ and $P( v_{\tau_j} - u) = 0$, Theorem \ref{thm: Energy1} implies that
\[
	v_{\tau_j} \to u
\]
in $C^0((0, \e), H^{2m}(\H)) \cap C^1((0, \e), H^{2m-1}(\H))$ as $j \to \infty$. 
In particular, Lemma \ref{le: comparison_sequence} implies the estimate
\begin{align*}
	&\norm{(u - w^N)(t, \cdot)}_{2m} + \sqrt t \norm{\n_t(u - w^N)(t, \cdot)}_{2m-1} \\
	&\quad = \lim_{j \to \infty} \left( \norm{(v_{\tau_j} - w^N)(t, \cdot)}_{2m} + \sqrt t \norm{\n_t(v_{\tau_j} - w^N)(t, \cdot)}_{2m-1} \right) \\
	&\quad \leq Ct^{N + \frac32}
\end{align*}
for each $t \in (0, \e)$. 
Since $N > 0$, we can let $t \to 0$ and conclude that $u \in C^0([0, \e), H^{2m}(\H)) \cap C^1([0, \e), H^{2m-1}(\H))$.
\end{proof}

Up to now we have kept $m, N$ fixed.
In the next lemma, we show that the obtained solution is actually independent of the choice of $m$ and $N$.

\begin{lemma}[Regularity of $u$] \label{le: localexistence}
There exists a unique $u \in C^\infty([0, \e) \times \H)$ such that
\begin{align*}
	Pu &= f, \\
	\n^k u|_{t = 0} &= \n^k w^N|_{t = 0},
\end{align*}
for all $N \in \N$ and $k \leq N$.
\end{lemma}

\begin{proof}
Choose an $\tilde m > m$ and an $\tilde N \in \N$ such that $\tilde N > D_{\tilde m}$ and $\tilde N > N$.
By the previous lemma, there is a section
\[
	\tilde u \in C^0([0, \e), H^{2\tilde m}(\H)) \cap C^1([0, \e), H^{2\tilde m-1}(\H))
\]
such that
\[
	P\tilde u = f
\]
and there is a constant $\tilde C > 0$ such that
\[
	\norm{(\tilde u - w^{\tilde N})(t, \cdot)}_{2m} + \sqrt t \norm{\n_t(\tilde u - w^{\tilde N})(t, \cdot)}_{2m-1} \leq \tilde Ct^{\tilde N + \frac32}.
\]
The first step is to show that $u = \tilde u$.
We have
\begin{equation} \label{eq: P u tilde u}
	P (\tilde u - u) = f - f = 0
\end{equation}
and $\tilde u - u \in C^0([0, \e), H^{2m}(\H)) \cap C^1([0, \e), H^{2m-1}(\H))$.
By the assumption in Theorem \ref{mainthm1} 
\[
	(\n_t)^k(w^{\tilde N} - w^N)|_{t = 0}
\]
for all $k \leq N$, which implies the estimate
\begin{align}
	\norm{(\tilde u - u)(t, \cdot)}_{2m} + &\sqrt t \norm{\n_t(\tilde u - u)(t, \cdot)}_{2m-1} \nonumber \\*
	& \leq \norm{(\tilde u - w^{\tilde N})(t, \cdot)}_{2m} + \sqrt t \norm{\n_t(\tilde u - w^{\tilde N})(t, \cdot)}_{2m-1} \nonumber\\*
	& \quad + \norm{(u - w^N)(t, \cdot)}_{2m} + \sqrt t \norm{\n_t(u - w^N)(t, \cdot)}_{2m-1} \nonumber \\*
	& \quad + \norm{(w^{\tilde N} - w^N)(t, \cdot)}_{2m} + \sqrt t \norm{\n_t(w^{\tilde N} - w^N)(t, \cdot)}_{2m-1} \nonumber \\
	& \leq C_1t^{N+1}\label{eq: estimate u tilde u}
\end{align}
for some constant $C_1 > 0$.
Since $C^\infty([0, \e) \times \H)$ is dense in $C^0((0, \e), H^{2m}(\H)) \cap C^1((0, \e), H^{2m-1}(\H))$ and the energy is continuous in the respective norms, Theorem \ref{thm: Energy1} applies also for sections in $C^0((0, \e), H^{2m}(\H)) \cap C^1((0, \e), H^{2m-1}(\H))$.
Theorem \ref{thm: Energy1}, equation \eqref{eq: P u tilde u} and estimate \eqref{eq: estimate u tilde u} imply that
\begin{align*}
\norm{(\tilde u - u)(t_1, \cdot)}_{2m} + &\sqrt {t_1} \norm{\n_t(\tilde u - u)(t_1, \cdot)}_{2m-1} \\
	& \leq D_m \left( \frac{t_1}{t_0} \right)^{D_m}\left( \norm{(\tilde u - u)(t_0, \cdot)}_{2m} + \sqrt {t_0} \norm{\n_t(\tilde u - u)(t_0, \cdot)}_{2m-1} \right) \\
	& \quad + D_m {t_1}^{D_m}\int_{t_0}^{t_1} \frac{\norm{P(\tilde u - u)(t, \cdot)}_{2m}}{t^{D_m + 1/2}}dt \\
	& \leq C_1D_m {t_1}^{D_m}{t_0}^{N + 1 - D_m}.
\end{align*}
Since $N > D_m$, we can let $t_0 \to 0$ and conclude that
\[
	\norm{(\tilde u - u)(t_1, \cdot)}_{2m} + \sqrt {t_1} \norm{\n_t(\tilde u - u)(t_1, \cdot)}_{2m-1} = 0,
\]
for all $t_1 \geq 0$.
Therefore $u = \tilde u \in C^0([0, \e), H^{2\tilde m}(\H)) \cap C^1([0, \e), H^{2\tilde m-1}(\H))$. 
This proves in particular the uniqueness part of the lemma.
Since $\tilde m$ was arbitrarily large, we conclude that
\[
	u \in C^0([0, \e), H^{2\tilde m}(\H)) \cap C^1([0, \e), H^{2\tilde m-1}(\H))
\]
for all large enough $\tilde m \in \N$.
By the previous lemma and since $u = \tilde u$, we have the estimate
\begin{equation} \label{eq: close_to_zero}
	\norm{(u - w^{\tilde N})(t, \cdot)}_{2\tilde m} + \sqrt t \norm{\n_t(u - w^{\tilde N})(t, \cdot)}_{2\tilde m-1} \leq C_2t^{\tilde N + \frac32}
\end{equation}
for any $\tilde m, \tilde N \in \N$ such that $\tilde m > m$ and $\tilde N > D_{\tilde m}$ and $\tilde N > N$ and $t \in (0, \e)$. 
The constant $C_2 > 0$ depends on $\tilde m$ and $\tilde N$.

By Proposition \ref{prop: time_function}, we can write $P = \psi \n_t^2 + L_1 \n_t + L_2$, where $L_1$ and $L_2$ are differential operators of first and second order respectively, which only differentiate in $\H_t$-direction. 
Since $\psi > 0$ for $t> 0$, we conclude that
\begin{equation} \label{eq: t_der}
	\n_t^2 u = \frac{1}{\psi} \left(f - L_1\n_t u - L_2 u\right)
\end{equation}
for all $t \in (0, \e)$.
This shows that $u \in C^2((0,\e), H^{m}(\H))$ for all $m \in \N$. 
Differentiating equation \eqref{eq: t_der}, one concludes that $u \in C^j((0, \e), H^m(\H))$ for all $j, m \in \N$. 
By the Sobolev embedding theorem, we conclude that $u \in C^\infty((0, \e) \times \H)$. 

What is missing is the regularity at the Cauchy horizon $t = 0$, since we cannot evaluate equation \eqref{eq: t_der} at $t = 0$.
We will use the equation
\begin{equation} \label{eq: higher_time_linear}
	\psi \n_t^2(u - w^N) = f - Pw^N - L_1 \n_t (u - w^N) - L_2 (u - w^N).
\end{equation}
From the assumptions in Theorem \ref{mainthm1}, there is a constant $C_3 > 0$ such that
\[
	\norm{(f - P w^N)(t, \cdot)}_{2m} \leq C_3t^{N +1}.
\]
Combining this with the estimate \eqref{eq: close_to_zero} shows that there is a constant $C_4 > 0$ such that
\[
	\norm{\psi \n_t^2(u - w^N)(t, \cdot)}_{2m-1} \leq C_4t^{N + 1}.
\]
This implies that
\[
	\norm{\n_t^2(u - w^N)(t, \cdot)}_{2m-1} \leq C_5t^{N},
\]
for some constant $C_5 > 0$. 
Since $w^N \in C^\infty([0, \e) \times \H)$, this shows that $u \in C^2([0, \e), H^m(\H))$ for all $m \in \N$.
Differentiating equation \eqref{eq: higher_time_linear} and using the obtained estimate on $\n_t^2(u - w^N)$ shows that $u \in C^3([0, \e), H^{m}(\H))$ for all $m \in \N$.
Iterating this and increasing $m$ and $N$ whenever necessary shows that $u \in C^j([0, \e), H^{m}(\H))$ for all $j, m \in \N$.
The Sobolev embedding theorem implies that $u \in C^\infty([0, \e) \times \H)$ as claimed.
\end{proof}

We finish the proof by going from the local to global existence and uniqueness.

\begin{proof}[Finishing the proof of Theorem \ref{mainthm1}]
By Lemma \ref{le: localexistence}, there is a unique solution $\tilde u$ on an open neighbourhood $U \cong [0, \e) \times \H$ of $\H$ in $\H \sqcup D(\S)$.
This neighbourhood contains a Cauchy hypersurface $\H_\tau$ of $D(\S)$. 
By Theorem \cite{BaerGinouxPfaeffle2007}*{Thm. 3.2.11}, we can now solve the Cauchy problem 
\begin{align*}
	P \hat u &= f \\
		\hat u|_{\H_\tau} &= \tilde u|_{\H_\tau} \\
		\n_t \hat u|_{\H_\tau} &= \n_t \tilde u|_{\H_\tau}
\end{align*}
on $D(\S)$.
Since $U \cap D(\S)$ is globally hyperbolic, it follows from Theorem \cite{BaerGinouxPfaeffle2007}*{Thm. 3.2.11} that $\hat u|_{U \cap D(\S)} = \tilde u|_{U \cap D(\S)}$. 
Pasting $\hat u$ and $\tilde u$ together gives the unique globally defined solution $u \in C^\infty(\H \sqcup D(\S))$. 
\end{proof}

\section{Proof of uniqueness for admissible wave equations} \label{sec: Uniqueness}

The idea for the proof of Theorem \ref{mainthm2} is to show that a solution to an admissible linear wave equation with trivial initial data must vanish up to any order and then apply Corollary \ref{cor: unique_cont}.

The following lemma is fundamental for the rest of this paper.
\begin{lemma} \label{le: iterative}
Let $P$ be a wave operator and write it as
\[
	P = \n^*\n + B(\n) + A
\]
for some connection $\n$ on $F$. Then for all $k \in \N$, we have
\begin{align}
	\n_t^kPu|_{t = 0} &= 2\n_V\n_t^{k+1} u|_{t = 0} + 2(k+1)\n_t^{k+1} u|_{t = 0} \nonumber \\
		& \quad - B(g(V, \cdot) \otimes \n_t^{k+1}u)|_{t = 0} + T_k(u|_{t = 0}, \n_t u|_{t = 0}, \hdots, \n_t^k u|_{t = 0}) \label{eq: iterative}
\end{align}
where $T_k$ is a linear differential operator, only differentiating in $\H$-direction.
\end{lemma}
\begin{proof}
By Proposition \ref{prop: time_function}, the metric is given by
\[
	g_{\a \b} = \begin{pmatrix}
	0 & 1 & 0 \\
	1 & -\psi & 0 \\
	0 & 0 & \g_{ij}
	\end{pmatrix}
	\Rightarrow
	g^{\a \b} = \begin{pmatrix}
	\psi & 1 & 0 \\
	1 & 0 & 0 \\
	0 & 0 & \g^{ij}
	\end{pmatrix},
\]
for $i,j \geq 2$, in the basis $(\d_t, \grad(t), e_2, \hdots, e_n)$, where $e_2, \hdots, e_n \in E$.
Recall from Section \ref{sec: null time function} that $\grad(t) = -\psi \d_t + Z$, we have
\begin{align*}
	\n^*\n u &= - \psi {\n_t}^2u - \n^2_{\d_t, \grad(t)}u - \n^2_{\grad(t), \d_t} u - \sum_{i,j \geq 2}\g^{ij}\n^2_{e_i, e_j} u \\
	&= - \psi {\n_t}^2u - 2\n_t\n_{\grad(t)}u + 2 \n_{\n_t\grad(t)}u - R^\n(\grad(t), \d_t)u - \sum_{i,j \geq 2}\g^{ij}\n^2_{e_i, e_j} u \\
	&= \psi {\n_t}^2u - 2\n_t \n_Zu + 2 (\d_t \psi)\n_tu + 2 \n_{\n_t\grad(t)}u - R^\n(\grad(t), \d_t)u - \sum_{i,j \geq 2}\g^{ij}\n^2_{e_i, e_j} u.
\end{align*}
Since $g(\n_t\grad(t), \d_t) = 0$ and  $g(\n_t\grad(t), \grad(t)) = - \frac12 \d_t \psi$, we have
\[
	\n_t\grad(t) = - \frac12 (\d_t \psi) \d_t + \frac 12 X,
\]
for some smooth vector field $X$ such that $X|_{\H_t} \in T\H_t$ for all $t \in [0, \e)$.
Recalling that $Z \in T\H_t$ for all $t \in [0, \e)$ gives
\begin{align*}
	\n^*\n u &= \psi {\n_t}^2u - 2\n_{Z}\n_tu + (\d_t\psi) \n_tu + R(u),
\end{align*}
where $R$ is a linear differential operator, only differentiating along $\H_t$.
Since $\d_t\psi(0, \cdot) = 2$ and $Z_{t = 0} = - V$, we conclude that
\[
	\n^k_{t} \n^*\n u|_{t = 0} = 2\n_V\n_t^{k+1}u|_{t = 0} + 2(k+1)\n_t^{k+1} u|_{t = 0} + R_k(u|_{t = 0}, \n_t u|_{t = 0}, \hdots, \n_t^k u|_{t = 0}),
\]
for every $k \in \N$, where $R_k$ is a linear differential operator, only differentiating in $\H$-direction.
Since
\begin{align*}
	B(\n u) &= \sum_{\a, \b = 0}^n g^{\a\b} B(g(e_\a, \cdot) \otimes \n_{e_\b}u) \nonumber \\
		&= \psi B(g(\d_t, \cdot) \otimes \n_tu) + B(g(\grad(t), \cdot) \otimes \n_tu) + B(g(\d_t, \cdot) \otimes \n_{\grad(t)}u) \nonumber \\*
	& \quad + \sum_{i,j = 2}^n \g^{ij}B(g(e_i, \cdot) \otimes \n_{e_j}u),
\end{align*}
we conclude by similar arguments that
\[
	\n^k_tB(\n u)|_{t = 0} = - B(g(V, \cdot) \otimes \n^{k+1}_t u)|_{t = 0} + S_k(u|_{t = 0}, \n_t u|_{t = 0}, \hdots, \n^k_t u|_{t = 0}),
\]
where $S_k$ is a linear differential operator, only differentiating in $\H$-direction.
It follows that
\begin{align*}
	\n^k_t(Pu)|_{t = 0} 
		&= \n^k_t(\n^* \n u + B(\n u) + A u)|_{t = 0} \\
		&=2\n_V\n_t^{k+1} u|_{t = 0} + 2(k+1)\n_t^{k+1} u|_{t = 0} \\
		& \quad - B(g(V, \cdot) \otimes \n_t^{k+1}u) + T_k|_{t = 0}(u, \n_t u, \hdots, \n_t^k u)
\end{align*}
where $T_k$ is a differential operator, only differentiating in $\H$-direction.
\end{proof}

\begin{proof}[Proof of Theorem \ref{mainthm2}]
We know that $u|_{t = 0} = 0$ by assumption. 
The goal is to show that 
\[
	\n_t^k u|_{t = 0} = 0
\]
for all $k \in \N$, since Corollary \ref{cor: unique_cont} would then imply the conclusion.

Assume we know that $u|_{t = 0} = \hdots = \n^k_t u|_{t = 0} = 0$. 
Lemma \ref{le: iterative} implies that
\[
	0 = 2\n_V\n_t^{k+1} u|_{t = 0} + 2(k+1)\n_t^{k+1} u|_{t = 0} - B(g(V, \cdot) \otimes \n_t^{k+1}u)|_{t = 0}.
\]
This implies that
\begin{align*}
	\d_V a(\n^{k+1}_t u, \n^{k+1}_t u)|_{t = 0}
		&= (\n_Va)(\n^{k+1}_t u, \n^{k+1}_t u) + 2 a(\n_V\n^{k+1}_t u, \n^{k+1}_t u)|_{t = 0} \\
		&= (\n_Va)(\n^{k+1}_t u, \n^{k+1}_t u) + a(B(g(V, \cdot) \otimes \n^{k+1}_t u, \n^{k+1}_t u)|_{t = 0} \\
		& \quad - 2(k+1) a(\n^{k+1}_t u, \n^{k+1}_t u)|_{t = 0} \\
		& \leq - 2(k+1) a(\n^{k+1}_t u, \n^{k+1}_t u)|_{t = 0},
\end{align*}
by the assumptions on $a$ and $B$.
Since $\H$ is compact, $a(\n^{k+1}_t u, \n^{k+1}_t u)|_{t = 0}$ must attain its maximum. 
In the maximum point, $\d_V a(\n^{k+1}_t u, \n^{k+1}_t u)|_{t = 0} = 0$ and therefore by the previous calculations
\[
	0 \leq -2(k+1) a(\n^{k+1}_t u, \n^{k+1}_t u)|_{t = 0}
\]
in the maximum point.
This implies that $a(\n^{k+1}_t u, \n^{k+1}_t u)|_{t = 0} = 0$ in the maximum point. 
Since $a$ is positive definite, $a(\n^{k+1}_t u, \n^{k+1}_t u)|_{t = 0}$ is non-negative and we conclude that
\[
	a(\n^{k+1}_t u, \n^{k+1}_t u)|_{t = 0} = 0
\]
everywhere.
Again, since $a$ is positive definite, we conclude that $\n^{k+1}_t u|_{t = 0} = 0$. 
By induction, this shows that $\n^k_t u|_{t = 0} = 0$ for all $k \in \N$ and Corollary \ref{cor: unique_cont} implies the statement.
\end{proof}

\section{Proof of existence for admissible linear scalar wave equations} \label{sec: linear_characteristic}

The goal of this section is to prove Theorem \ref{mainthm3}.
The proof is based on the fact that we may compute the asymptotic solution rather explicitly for admissible scalar wave equations and therefore apply Theorem \ref{mainthm1}. 
We start by proving Lemma \ref{le: scalar_admissible}.

\begin{proof}[Proof of Lemma \ref{le: scalar_admissible}]
The trivial connection $\n := \d$ is of course metric with respect to pointwise multiplication.
Therefore condition \eqref{eq: admissible} is equivalent to the condition 
\begin{equation}\label{eq: scalar_admissible}
	g(V, W|_\H) \leq 0.
\end{equation}
There is a smooth function $\beta$ such that $W|_\H - \beta \d_t|_\H \in T\H$.
Since $g(V, \d_t|_\H) = -1$, inequality \eqref{eq: scalar_admissible} becomes
\[
	\beta \geq 0.
\]
This is satisfied if and only if $W|_\H$ is nowhere outward pointing. 
\end{proof}

It is clear that Lemma \ref{le: scalar_admissible} and Theorem \ref{mainthm2} imply the uniqueness part of Theorem \ref{mainthm3}.
The following lemma is crucial for the existence part of Theorem \ref{mainthm3}.

\begin{lemma} \label{le: ODE_Riem}
Let $(K, h)$ be a compact Riemannian manifold (without boundary) and assume that $X$ is a nowhere vanishing Killing vector field with respect to $h$, i.e.
\[
	\L_X h = 0.
\]
Let $\alpha \in C^\infty(K)$ be a smooth nowhere vanishing function.
Then
\begin{align*}
		C^\infty(K) &\to C^\infty(K) \\
		u &\mapsto \partial_X u + \a u
\end{align*}
is an isomorphism of topological vector spaces. 
\end{lemma}
\begin{proof}
First we show injectivity. 
Assume therefore that $\d_X u + \a u = 0$. 
Since $K$ is compact, the maximum and minimum values of $u$ are attained at, let us say, $x_{\max}$ and $ x_{\min}$. 
At these points $\partial_Xu(x_{\min}) = 0 = \partial_X u(x_{\max})$. 
Since $\a$ is nowhere vanishing, it follows that $u(x_{\min}) = 0 = u(x_{\max})$. 
Hence $u = 0$, which proves injectivity.

To prove surjectivity, let $f \in C^\infty(K)$ be given. 
We assume that $\a > 0$, the case $\a < 0$ is similar since $\L_{-X}h = - \L_X h = 0$.
For each point $p \in K$, let $\phi_s(p)$ denote the flow of $p$ along $X$, i.e. 
\begin{align*}
	\phi_0(p) &= p, \\
	\der s{s_0} \phi_s(p) &= X|_{\phi_{s_0}(p)}.
\end{align*}
Since $K$ is compact, the flow exists for all $s \in \R$.
We claim that the function $u$ defined for each point $p \in K$ by
\begin{equation} \label{eq: explicit_ODE}
	u(p) := \int_{-\infty}^0e^{- \int_s^0 \alpha \circ \phi_a(p)da} f \circ \phi_s(p)ds
\end{equation}
is smooth and solves $\d_X u + \a u = f$. 
Since $\a$ is bounded from below by a positive constant, $u$ defined by equation \eqref{eq: explicit_ODE} is well-defined.
The key point in showing smoothness of $u$ is that 
\begin{align*}
	(K, h) &\to (K, h) \\
	p &\mapsto \phi_s(p)
\end{align*}
is an \emph{isometry} for all $s \in \R$. 
In other words, we will use that
\begin{equation} \label{eq: h_isometry}
	\abs{d\phi_s(Y)}_{h} = \abs{Y}_h,
\end{equation}
for all $Y \in T\H$ and $s \in \R$.
The idea is to apply the Lebesgue dominated convergence theorem. 
For this, we need to make sure that the partial derivatives of the integrand are integrable.
For each smooth vector field $Y$ and $s \leq 0$, we have
\begin{align*}
	& \abs{\d_Y \left(e^{-\int_s^0 \a \circ \phi_a(p) da}f \circ \phi_s(p)\right)} \\
		&\leq e^{-\int_s^0 \a \circ \phi_a(p) da} \left(\norm{f}_{C^0}\int_s^0 \abs{d\a \circ d\phi_a(Y)}da + \abs{df \circ d\phi_s(Y)}\right) \\*
		& \leq e^{-\int_s^0 \a \circ \phi_a(p) da} \left(\norm{f}_{C^0} \norm{\a}_{C^1} \int_s^0 \abs{d\phi_a(Y)}_h da + \norm{f}_{C^1}\abs{d\phi_s(Y)}_h \right) \\*
		& \leq e^{-\int_s^0 \a \circ \phi_a(p) da} \abs{Y}_h (\norm{f}_{C^0}\abs s \norm{\a}_{C^1} + \norm{f}_{C^1}),
\end{align*}
where the $C^k$-norms are defined with respect to the metric $h$.
Since $\a$ is bounded from below by a positive constant, we conclude that
\[
	\abs{\d_Y \left(e^{\int_0^s \a \circ \phi_a(p) da}f \circ \phi_s(p)\right)}_h \leq (1 + \abs s)Ce^{Cs}\abs{Y}_h,
\]
for some constant $C > 0$, for all $s \leq 0$.
Since the right hand side is in $L^1(- \infty, 0)$, we may apply Lebesgue's dominated convergence theorem to equation \eqref{eq: explicit_ODE} and conclude that $u \in C^1$.
To show that the second derivative of the integrand is integrable, we use the fact that $\hat \n d\phi_s = 0$, where $\hat \n$ is the Levi-Civita connection with respect to $h$, since $\phi_s$ are isometries.
We only compute one relevant term, the other term is treated analogously.
\begin{align*}
	\abs{\Hess(f \circ \phi_s)(Y, Z)}_h 
		&= \abs{(\d_Y\d_Z - \d_{\hat \n_YZ})f \circ \phi_s)}_h \\
		&= \abs{\d_Y df \circ d\phi_s(Z) -df \circ d\phi_s(\hat \n_YZ)}_h \\
	 	&= \abs{\hat \n df(d\phi_s(Y), d\phi_s(Z)) + df \circ \hat \n d\phi_s(Y, Z)}_h \\
	 	&= \abs{\hat \n df(d\phi_s(Y), d\phi_s(Z))}_h \\
	 	&\leq \norm{f}_{C^2} \abs{d\phi_s(Y)}_h \abs{d\phi_s(Z)}_h \\
	 	&= \norm{f}_{C^2} \abs{Y}_h \abs{Z}_h,
\end{align*}
independently of $s$.
Again, by the Lebesgue dominated convergence theorem, we conclude that $u \in C^2$.
Since $f$ and $\a$ and all derivatives are uniformly bounded on $K$, one iterates this to show that $u$ is in fact smooth.
Furthermore, we calculate that
\begin{align*}
	\partial_X u(p) 
		&= \frac{d}{dt}\Big|_{t = 0} \int_{-\infty}^0 e^{-\int_s^0 \a \circ \phi_{a}(\phi_t(p))da} f \circ \phi_s(\phi_t(p)) ds \\*
		&= \frac{d}{dt}\Big|_{t = 0} \int_{-\infty}^0 e^{-\int_s^0 \a \circ \phi_{a+t}(p)da} f \circ \phi_{s+t}(p) ds \\*
		&= \frac{d}{dt}\Big|_{t = 0} e^{- \int_0^t \a \circ \phi_b(p) db}\int_{-\infty}^t e^{-\int_w^0 \a \circ \phi_{b}(p)db} f \circ \phi_w(p) dw \\*
		&= - \a(p) u(p) + f(p).
\end{align*}
Hence $\d_X u + \a u = f$ as claimed. 
\end{proof}

%\begin{remark}\label{rmk: positivity}
%Note from the explicit formula \eqref{eq: explicit_ODE} that if the right hand side $f$ is nowhere zero, then it follows that the solution $u$ is nowhere zero as well.
%\end{remark}

Using this, we now prove Theorem \ref{mainthm3}.

\begin{proof}[Proof of Theorem \ref{mainthm3}]
Uniqueness of solution follows by Lemma \ref{le: scalar_admissible} and Theorem \ref{mainthm2}.

Theorem \ref{mainthm1} implies that it suffices to compute an asymptotic solutions $w^N$ in order to prove existence of a solution.
Let us make the ansatz 
\[
	w^N(t,x) := \sum_{j = 0}^{N+1} \frac{u_j(x)}{j!}t^j,
\]
where $(u_j)_{j=0}^\infty \subset C^\infty(\H)$.
By Lemma \ref{le: iterative}, we know that $(\d_t)^k\left(Pw^N - f\right)|_{t = 0} = 0$ is equivalent to
\[
	0 = 2\d_V u_{k+1} + 2(k+1)u_{k+1} - g(V, W)|_{t = 0}u_{k+1} + T_k|_{t = 0}(u_0, \hdots, u_k) - \d_t^kf|_{t = 0}.
\]
Since $P$ is an admissible wave operator, there is a \emph{non-negative} function $\b \in C^\infty(\H)$ such that $P|_{t = 0} - (\Box + \b \d_t)|_{t = 0}$ is a differential operator of first order, only differentiating in $\H$-direction. 
It follows that $g(W, V)|_{t = 0} = g(\b\d_t, V) = - \b$. 
This implies that
\begin{equation} \label{eq: iterative_asymptotic}
	\d_V u_{k+1} + (k+1 + \frac12 \beta)u_{k+1} = \frac12 T_k|_{t = 0}(u_0, \hdots, u_k) - \frac12 \d_t^kf|_{t = 0}.
\end{equation}

We want to solve \eqref{eq: iterative_asymptotic} iteratively using the previous lemma.
For this, we need to show that $\L_V \sigma|_{t = 0} = 0$, where $\sigma$ was defined in \eqref{eq: sigma}.
Since $\H$ is totally geodesic, it follows that $\L_Vg|_{t = 0}(X, Y) = 0$ for all $X, Y \in T\H$. 
We have
\begin{align*}
	\L_V \sigma|_{t = 0}(X,Y) &= \L_V g|_{t = 0}(X, Y) + \d_V(g(X, \d_t)g(Y, \d_t))|_{t = 0} \\
		&\quad - g([V, X], \d_t)g(Y, \d_t)|_{t = 0} - g(X, \d_t)g([V, Y], \d_t)|_{t = 0} \\
		&= \d_V(g(X, \d_t)g(Y, \d_t))|_{t = 0} - g([V, X], \d_t)g(Y, \d_t)|_{t = 0} - g(X, \d_t)g([V, Y], \d_t)|_{t = 0}.
\end{align*}
Clearly $\L_V \sigma|_{t = 0} (V,V) = 0$ since $g(V, \d_t)|_{t = 0} = -1$. 
Moreover, $\L_V \sigma|_{t = 0}(e_i, e_j) = 0$, since $g(e_i, \d_t)|_{t = 0} = 0$ for all $i$.
In order to show that $\L_Vg|_{t = 0}(V, e_i) = 0$, we need to use that $\Ric(e_i, V)|_{t = 0} = 0$ by assumption.
By Proposition \ref{prop: time_function}, the metric is given by
\[
	g_{\a \b}|_{t = 0} = \begin{pmatrix}
	0 & -1 & 0 \\
	-1 & 0 & 0 \\
	0 & 0 & \delta_{ij}
	\end{pmatrix}
	= g^{\a \b}|_{t = 0},
\]
for $i,j \geq 2$, in the basis $(\d_t, V, e_2, \hdots, e_n)$. By assumption, we have
\begin{align*}
	0 	&= \Ric(e_i, V)|_{t = 0} \\
		&= -R(V, e_i, V, \d_t)|_{t = 0} + \sum_{j = 2}^n R(e_j, e_i, V, e_j)|_{t = 0} \\
		&= g(\n_{e_i} \n_V V, \d_t)|_{t = 0} - g(\n_V \n_{e_i} V, \d_t)|_{t = 0} - g(\n_{[e_i, V]} V, \d_t)|_{t = 0}  \\
		&\quad +  \sum_{j = 2}^n \left( g(\n_{e_j} \n_{e_i} V, e_j)|_{t = 0} - g(\n_{e_i} \n_{e_j} V, e_j)|_{t = 0} - g(\n_{[e_j, e_i]} V, e_j)|_{t = 0} \right) \\
		&= g(\n_{e_i}V, \d_t)|_{t = 0} + g(\n_{\n_V e_i} V, \d_t)|_{t = 0} \\
		&= g(\n_{\n_V e_i}V, \d_t)|_{t = 0},
\end{align*}
where we have used that $\n_V V = V$ and $\n_{e_i}V = 0$ on $\H$ and that $\H$ is totally geodesic. 
Again since $\H$ is totally geodesic, we have $\n_V e_i \in T\H$ and may conclude that
\[
	g(\n_{\n_V e_i}V, X)|_{t = 0} = 0
\]
for all $X \in T \H$. 
Altogether it follows that $\n_{\n_Ve_i}V|_{t = 0} = 0$ and hence $[V, e_i] = \n_Ve_i|_{t = 0} \in E$. 
Now, this means that $g([V, e_i], \d_t)|_{t = 0} = 0$, which implies that also $\L_V \sigma|_{t = 0}(e_i, V) = 0$.

Since $k+1 + \frac{1}{2} \beta > 0$, we may apply Lemma \ref{le: ODE_Riem} with $K = \H$, $X = V$, $h = \sigma$ and $\a = k+1 + \frac{1}{2} \beta$ to iteratively solve equation \eqref{eq: iterative_asymptotic} in a unique way.
Since $u_0$ is given, we obtain in this way the asymptotic solutions $w^N$ for any $N \in \N$.

By Theorem \ref{mainthm1}, we conclude that the continuous map
\begin{align*}
	C^\infty(\H \sqcup D(\S)) &\to C^\infty(\H) \times C^\infty(\H \sqcup D(\S)) \\
		u &\mapsto (u|_\H, Pu)
\end{align*}
bijective. 
The open mapping theorem for Fréchet spaces now implies that the inverse is continuous as well, i.e.\ the solution depends continuously on $u_0$ and $f$.
\end{proof}

\section{Proof of local existence for admissible non-linear scalar wave equations} \label{sec: non-linear_characteristic}

The goal of this section is to prove Theorem \ref{mainthm4}.
We will perform a Picard-type iteration starting with the asymptotic solution.
An argument using the energy estimate and the asymptotic solution implies that there is a limit, which is smooth up to the Cauchy horizon on a small neighbourhood of the Cauchy horizon.

Let $u_0$ and $f$ be given by the assumptions in Theorem \ref{mainthm4}.
Similarly to when we proved existence of solutions for the linear wave equation, we will need an \emph{asymptotic solution} for the non-linear wave equation.

\begin{lemma}[The asymptotic solution] \label{le: non-linear_asymptotic_expansion} 
There are functions $(u_j)_{j \in \N} \subset C^\infty(\H)$ such that for each $N \in \N$, the function $w^N \in C^\infty([0, \e) \times \H)$ defined by
\[
	w^N(x,t) := \sum_{j=0}^{N+1} \frac{u_j(x)t^j}{j!}
\]
satisfies
\begin{align*}
	(\d_t)^k(P w^N - f(w^N))|_{t = 0} &= 0, \\
		w^N|_{t = 0} &= u_0,
\end{align*}
for all $k \leq N$.
\end{lemma}
\begin{proof}
By Lemma \ref{le: iterative} it follows that 
\[
	(\d_t)^k(P w^N - f(w^N))|_{t = 0} = 0
\]
is equivalent to
\begin{equation} \label{eq: defining_u_n}
	\d_Vu_{k+1} + \left(k+1 + \frac{1}{2} \beta\right)u_{k+1}= Q_k(u_0, \hdots, u_k) + \frac12 (\d_t)^kf(w^N)|_{t = 0},
\end{equation}
for some smooth non-negative function $\b$.
The right hand side only depends on 
\[
	u_0 = w^N|_{t = 0}, \hdots, u_k = (\d_t)^k w^N|_{t = 0}.
\] 
Since $k+1 + \frac12\b > 0$, the same argument as in the proof of Theorem \ref{mainthm3} implies that we may inductively define $u_{k+1}$ as the unique solutions to \eqref{eq: defining_u_n}.
\end{proof}

Recall that the goal is to show that there is a smooth solution $u$, defined on a small future neighbourhood of $\H$ such that
\begin{align*}
	Pu &= f(u), \\
	u|_\H &= u_0.
\end{align*}
We will construct this neighbourhood by showing that the solution $u$ exists on an open set of the form $[0, T) \times \H$ for some $T \in (0, \e)$ yet to be defined.

Let us fix $m \in N$ such that $2m \geq \frac{\dim(\H)}2$ and an $N \in \N$ such that $N > D_m$, where $D_m > 0$ is the constant from Theorem \ref{thm: Energy1}. 
The idea is again to construct a sequence of functions that converge to a solution $u$.
Define the sequence $(v_k)_{k \in \N} \subset C^\infty(\H \sqcup D(\S))$ by first choosing $v_0 := w^N$ and then iteratively solve the characteristic Cauchy problems
\begin{align*}
	P v_{k+1} &= f(v_k), \\
		v_{k+1}|_{\H} &= u_0,
\end{align*}
for each $k \in \N$.
Theorem \ref{mainthm3} implies that these characteristic Cauchy problems can be solved uniquely. 
\begin{lemma} \label{le: asymptotic_behaviour}
For each $k \in \N$ and $j \leq N+1$, we have
\begin{equation} \label{eq: v_k_u_n}
	(\d_t)^j v_k|_{t=0} = u_j.
\end{equation}
\end{lemma}
\begin{proof}
The case $k = 0$ is clear by construction.
Let us assume that \eqref{eq: v_k_u_n} holds for all $k' \leq k$ for some $k \in \N$ and show \eqref{eq: v_k_u_n} for $k + 1$. 
Similar to the proof of Lemma \ref{le: non-linear_asymptotic_expansion}, we know that $(\d_t)^j (Pv_{k+1} - f(v_k))|_{t = 0}  = 0$ is equivalent to
\begin{align}
	\d_V (\d_t)^{j+1} v_{k+1}|_{t = 0} + \left(j+1 + \frac{1}{2} \beta\right)(\d_t)^{j+1}v_{k+1}|_{t = 0} &= Q(v_{k+1}|_{t = 0}, \hdots, (\d_t)^jv_{k+1}|_{t = 0}) \nonumber \\
		&\quad + \frac12 (\d_t)^jf(v_k)|_{t = 0}. \label{eq: n+1_k+1}
\end{align}
By the induction assumption, we know that $(\d_t)^jf(v_k)|_{t = 0} = (\d_t)^jf(w^N)|_{t = 0}$ for all $j \leq N+1$.
Let us now, for fixed $k+1$, perform induction in $j$.
We know that $(\d_t)^0 v_{k+1}|_{t=0} - u_0 = 0$.
Now assume that we know that $(\d_t)^iv_{k+1}|_{t = 0} = u_i$ for all $i \leq j \leq N$.
By equation \eqref{eq: defining_u_n} in the proof of Lemma \ref{le: non-linear_asymptotic_expansion} and equation \eqref{eq: n+1_k+1}, we deduce that
\[
	\d_V ((\d_t)^{j+1} v_{k+1}|_{t = 0} - u_{j+1}) + \left(j+1 + \frac{1}{2} \beta\right)((\d_t)^{j+1}v_{k+1}|_{t = 0} - u_{j+1}) = 0.
\]
Lemma \ref{le: ODE_Riem} implies that $(\d_t)^{j+1}v_{k+1}|_{t = 0} = u_{j+1}$.
By induction in $j$, we conclude that $(\d_t)^j v_{k+1}|_{t = 0} = u_j$ for all $j \leq N+1$.
This completes the induction step in $k$ and concludes therefore the proof.
\end{proof}

Combining this with Theorem \ref{thm: Energy1}, we are able to deduce the following energy estimates.

\begin{lemma}[Energy estimate for the sequence] \label{le: energy_sequence}
For each $k, l \in \N$ and each $t \in (0, \e)$, we have
\begin{align*}
	\norm{(v_k - v_{l})(t, \cdot)}_{2m+1}	 + \sqrt t \norm{\d_t(v_k - v_{l})(t, \cdot)}_{2m} &\leq D_m t^{D_m} \int_0^t \frac{\norm{P(v_k - v_{l})(s, \cdot)}_{2m}}{s^{D_m + \frac12}} ds.
\end{align*}
\end{lemma}
\begin{proof}
Lemma \ref{le: asymptotic_behaviour} implies that
\begin{align*}
	\left( \frac{d}{dt}\right)^j \norm{v_k - v_l}_{2m+1}^2|_{t=0} = 0, \\
	\left( \frac{d}{dt}\right)^{j-1} \norm{\d_t (v_k - v_l)}_{2m}^2|_{t=0} = 0, 
\end{align*}
for all $j \leq 2N+3$. It follows that there are constants $C_{k,l} > 0$ such that
\[
	\norm{(v_k - v_l)(t, \cdot)}_{2m+1} + \sqrt t \norm{\d_t (v_k - v_l)(t, \cdot)}_{2m} \leq C_{k,l}t^{N + 2}.
\]
Applying Theorem \ref{thm: Energy1} implies that for any $t_0 < t$, we have
\begin{align*}
	&\norm{(v_k - v_l)(t, \cdot)}_{2m+1} + \sqrt t \norm{\d_t(v_k - v_l)(t, \cdot)}_{2m} \\*
	&	\quad \leq D_m \left(\frac{{t}}{{t_0}}\right)^{D_m}\left(\norm{(v_k - v_l)(t_0, \cdot)}_{2m+1} + \sqrt {t_0} \norm{\d_t(v_k - v_l)(t_0, \cdot)}_{2m}\right) \\
	& \quad \quad + D_m{t}^{D_m}\int_{t_0}^{t}\frac{\norm{P(v_k - v_l)(s, \cdot)}_{2m}}{s^{D_m + 1/2}}ds \\
	& \quad \leq D_mC_{k,l}t^{D_m} t_0^{N + 2 - D_m} + D_m{t}^{D_m}\int_{t_0}^{t}\frac{\norm{P (v_k - v_l)(s, \cdot)}_{2m}}{s^{D_m + \frac12}}ds.
\end{align*}
Since $N > D_m$, we can let $t_0 \to 0$ and conclude the statement.
\end{proof}

Let us compute bounds on the non-linearity $f$.

\begin{lemma} \label{le: non-linearity_estimate}
Let $[t_0, t_1] \subset [0, \e)$. 
For each constant $B_1 > 0$, there is a constant $B_2 > 0$ (depending on $m$) such that if $v \in C^\infty([t_0, t_1] \times \H)$ satisfies $\norm{v(t, \cdot)}_{2m} \leq B_1$ for all $t\in [t_0, t_1]$, then $\norm{f(v)(t, \cdot)}_{2m} \leq B_2$ for all $t \in [t_0, t_1]$.
\end{lemma}
\begin{proof}
Since $2m > \frac{\dim(\H)}{2}$, the Sobolev embedding theorem implies that there is a constant $C_{Sob} > 0$ such that
\begin{align*}
	\norm{v(t, \cdot)}_{\infty} &\leq C_{Sob}\norm{v(t, \cdot)}_{2m} \\
		&\leq C_{Sob} B_1
\end{align*}
for all $t \in [0,\epsilon)$.
Using this, we get the estimate
\begin{align*}
	\norm{f(v)(t,\cdot)}_{2m} & \leq C \sum_{j = 0}^{2m} \sup_{\abs{x} \leq C_{Sob}B_1}\norm{f^{(j)}(x)(t, \cdot)}_{\infty} \norm{v(t, \cdot)}_{2m}^j \nonumber \\
		& \leq  2m C \sup_{\abs{x} \leq C_{Sob}B_1}\norm{f(x)(t, \cdot)}_{C^{2m}}\max({B_1}^{2m}, 1),
\end{align*}
for some constant $C > 0$. 
Since $[-C_{Sob}B_1, C_{Sob}B_1] \times [t_0, t_1] \times \H$ is compact, we can define
\[
	B_2 := 2m C \sup_{\abs{x} \leq C_{Sob}B_1}\norm{f(x)(t, \cdot)}_{C^{2m}}\max({B_1}^{2m}, 1) < \infty,
\]
which concludes the proof.
\end{proof}

\begin{lemma}[Boundedness of the sequence] \label{le: bounded_sequence}
There is a $T \in (0, \e)$ and a constant $C > 0$ (depending on $m$ and $N$) such that
\[
	\norm{(v_k - w^N)(t, \cdot)}_{2m+1} + \sqrt t \norm{\d_t(v_k - w^N)(t, \cdot)}_{2m} \leq C t^{N + \frac32} 
\]
for all $k \in \N$.
In particular $v_k|_{[0,T] \times \H}$ is bounded in $C^0([0, T], H^{2m+1}(\H))$.
\end{lemma}
The constants $T$ and $C$ depend on $m$ and $N$.
\begin{proof}
For $k = 0$, the statement is trivially true.
Choosing $l = 0$ in Lemma \ref{le: energy_sequence}, we get the estimate
\begin{align}
	& \norm{(v_{k+1} - w^N)(t, \cdot)}_{2m+1} + \sqrt t \norm{\d_t(v_{k+1} - w^N)(t, \cdot)}_{2m} \nonumber \\*
		& \quad \leq D_m t^{D_m} \int_0^t \frac{\norm{P(v_{k+1} - w^N)(s, \cdot)}_{2m}}{s^{D_m + \frac12}} ds \nonumber \\*
		& \quad \leq D_m t^{D_m} \int_0^t \frac{\norm{(f(v_k) - f(w^N))(s, \cdot)}_{2m}}{s^{D_m + \frac12}} ds + D_m t^{D_m} \int_0^t \frac{\norm{(f(w^N) - Pw^N)(s, \cdot)}_{2m}}{s^{D_m + \frac12}} ds, \label{eq: bounded_seqeunce}
\end{align}
for each $t \in (0, \e)$ and each $k \in \N$. 
We first estimate the second term.
Lemma \ref{le: non-linear_asymptotic_expansion} implies that $\norm{(f(w^N) - Pw^N)(t, \cdot)}_{2m} \leq C_1t^{N+1}$ for some constant $C_1 > 0$. 
Since $N > D_m$, the integral in the second term in equation \eqref{eq: bounded_seqeunce} is bounded.
We calculate
\begin{align}
	D_m t^{D_m} \int_0^t \frac{\norm{(f(w^N) - Pw^N)(s, \cdot)}_{2m}}{s^{D_m + \frac12}} ds
		& \leq C_1 D_m t^{D_m} \int_0^t \frac{s^{N+1}}{s^{D_m + \frac12}} ds \nonumber \\*
		& = C_2 t^{N + \frac32} \label{eq: k=0_estimate}
\end{align}
for some constant $C_2 > 0$. 
It remains to estimate the first term in equation \eqref{eq: bounded_seqeunce}.
For the induction step, assume that for a $k \geq 0$, we have the estimate
\[
	\norm{(v_k - w^N)(t, \cdot)}_{2m+1} + \sqrt t \norm{\d_t(v_k - w^N)(t, \cdot)}_{2m} \leq Ct^{N + \frac32},
\]
for some fixed constant $C > C_2$, for all $t \in [0, T]$ and for some $T > 0$ yet to be chosen.
By equation \eqref{eq: k=0_estimate}, this holds for $k = 0$.
Since $2m > \frac{\dim(\H)}{2}$, there is a constant $C_3 > 0$ such that
\begin{align}
	& \norm{(f(v_k) - f(w^N))(t, \cdot)}_{2m} \nonumber \\
		&\qquad \leq \norm{(v_k - w^N)(t, \cdot) \int_0^1 f'(\tau v_k(t, \cdot) + (1-\tau)w^N(t, \cdot) )d\tau}_{2m} \nonumber \\
		&\qquad \leq C_3 \norm{(v_k - w^N)(t, \cdot)}_{2m} \int_0^1 \norm{f'(\tau v_k(t, \cdot) + (1-\tau)w^N(t, \cdot) )d\tau}_{2m}. \label{eq: f_estimate} \nonumber
\end{align}
Note that 
\begin{align*}
	\norm{\tau v_k(t, \cdot) + (1-\tau)w^N(t, \cdot)}_{2m} &\leq \tau \norm{v_k(t, \cdot) - w^N(t, \cdot)}_{2m} + \norm{w^N(t, \cdot)}_{2m}\\
		&\leq C \e^{N + \frac32} + \sup_{t \in [0, \e)}\norm{w^N(t, \cdot)}_{2m} \\
		&=: B_1
\end{align*}
which is a bound that is independent of $k$.
By Lemma \ref{le: non-linearity_estimate}, with $f$ replaced by $f'$, there is a constant $B_2 > 0$ such that 
\[
	\norm{f'(\tau v_k(t, \cdot) + (1-\tau)w^N(t, \cdot) )d\tau}_{2m} \leq B_2.
\]
We can now estimate the first term in equation \eqref{eq: bounded_seqeunce} as
\begin{align*}
	D_m t^{D_m} \int_0^t \frac{\norm{(f(v_k) - f(w^N))(s, \cdot)}_{2m}}{s^{D_m + \frac12}} ds 
		&\leq D_mt^{D_m}C_3 B_2\int_0^t \frac{\norm{(v_k - w^N)(s, \cdot)}_{2m}}{s^{D_m + \frac12}}ds \\
		&\leq  D_mt^{D_m}C_3 B_2 C \int_0^t \frac{s^{N + \frac32}}{s^{D_m + \frac12}}ds \\
		&\leq  \frac{D_mC_3 B_2 C}{N + 2 - D_m}t^{N + 2},
\end{align*}
where we have used that $N > D_m$.

Altogether, inequality \eqref{eq: bounded_seqeunce} becomes
\begin{align*}
	\norm{(v_{k+1} - w^N)(t, \cdot)}_{2m+1}  &+ \sqrt t \norm{\d_t(v_k - w^N)(t, \cdot)}_{2m} \\
		&\leq  t^{N + \frac32} \left( C_2 + \frac{D_mC_3 B_2 C}{N + 2 - D_m} \sqrt t \right).
\end{align*}
Since $C > C_2$, there is a $T > 0$ such that
\[
	C_2 + \frac{D_mC_3 B_2 C}{N + 2 - D_m} \sqrt T \leq C.
\]
This implies that
\[
	\norm{(v_{k+1} - w^N)(t, \cdot)}_{2m+1} \leq Ct^{N + \frac32} 
\]
for all $t \in [0, T]$. 
Since $T$ is independent of $k$, this concludes the assertion by induction.
\end{proof}

Let from now on $T$ be as in Lemma \ref{le: bounded_sequence}.

\begin{lemma}[Local existence] \label{le: convergence_sequence}
There is a 
\[
	u \in C^0([0, T], H^{2m+1}(\H)) \cap C^1([0, T], H^{2m}(\H))
\]
such that
\[
	v_k|_{[0, T] \times \H} \to u
\]
in $C^0([0, T], H^{2m+1}(\H)) \cap C^1([0, T], H^{2m}(\H))$. In particular,
\begin{align*}
	Pu &= f(u), \\
	u|_{t = 0} &= u_0.
\end{align*}
\end{lemma}
\begin{proof}
We need to show that $v_k|_{[0, T] \times \H}$ is a Cauchy sequence. 
By Lemma \ref{le: bounded_sequence} we know that $\sup_{t \in [0, T]}\left(\norm{\tau v_k(t, \cdot) + (1-\tau)v_{k-1}(t, \cdot)}_{2m}\right)$ is uniformly bounded in $k \in \N$ and $\tau \in [0, 1]$. 
Using this, Lemma \ref{le: non-linearity_estimate} and that $2m > \frac{\dim(\H)}{2}$, we conclude that
\begin{align*}
	\norm{f(v_k) - f(v_{k-1})}_{2m} &= \norm{(v_k - v_{k-1}) \int_0^1 f'(\tau v_k +(1-\tau) v_{k-1}) d\tau}_{2m} \nonumber \\
		& \leq C_1 \norm{(v_k- v_{k-1})}_{2m} \int_0^1 \norm{f'(\tau v_k +(1-\tau) v_{k-1})}_{2m}d\tau \\
		& \leq C_2 \norm{(v_k - v_{k-1})(t, \cdot)}_{2m}, 
\end{align*}
for some constant $C_2 > 0$ independent of $k$. 
Define 
\[
	A_k(t) := \norm{(v_{k+1} - v_{k})(t, \cdot)}_{2m+1}	 + \sqrt t \norm{\d_t(v_{k+1} - v_{k})(t, \cdot)}_{2m}.
\]
Lemma \ref{le: energy_sequence} implies now the recursive relation
\begin{align*}
	A_k(t) & \leq D_m t^{D_m} \int_0^t \frac{\norm{(f(v_{k}) - f(v_{k-1}))(s, \cdot)}_{2m}}{s^{D_m + \frac12}} ds \\
		&\leq C_3 t^{D_m} \int_0^t \frac{A_{k-1}(s)}{s^{D_m + \frac12}} ds,
\end{align*}
where $C_3$ depends on $m$, but not on $k$.
We get by iteration
\[
	A_k(t) \leq t^{D_m} (C_3)^k \int_0^t \frac{1}{\sqrt{s_{k-1}}} \hdots \int_{0}^{s_2} \frac{1}{\sqrt {s_{1}}} \int_0^{s_1} \frac{A_0(s_0)}{(s_0)^{D_m + \frac12}}ds_0 \hdots ds_{k-1}.
\]
for all $t \in [0, T]$. 
By Lemma \ref{le: bounded_sequence} 
\[
	A_0(t) \leq C_4t^{N + \frac32}
\]
for some constant $C_4 > 0$.
Since $N > D_m$, we may estimate 
\[
	\frac{A_0(t)}{t^{D_m}} \leq C_5
\]
for all $t \in [0, T]$.
This simplifies the computation of the integral to
\begin{align*}
	A_k(t) & \leq (C_3)^k C_4C_5 t^{D_m} \int_0^t \frac{1}{\sqrt{s_{k-1}}} \hdots \int_{0}^{s_2} \frac{1}{\sqrt {s_{1}}} \int_0^{s_1} \frac{1}{\sqrt{s_0}}ds_0 \hdots ds_{k-1} \\
			&\leq C_6 \frac{(2\sqrt T C_3)^k}{k!},
\end{align*}
for all $t \in [0, T]$, for some constant $C_6 > 0$.
We conclude the estimate
\begin{align}
	\sup_{t \in [0,T]} \left(\norm{(v_{k+j} - v_k)(t, \cdot)}_{2m+1} + \sqrt t \norm{\d_t(v_{k+j} - v_k)(t, \cdot)}_{2m} \right) &\leq \sum_{i = 0}^{j-1} \sup_{t \in [0, T]}A_{k+i}(t) \nonumber \\
	& \leq C_6 \sum_{i = 0}^{j-1}\frac{(2\sqrt T C_3)^{k+i}}{(k+i)!}.
\end{align}
This implies that $v_k|_{[0, T] \times \H}$ is a Cauchy sequence in the Banach space 
\[
	C^0([0, T], H^{2m+1}(\H)) \cap C^1([0, T], H^{2m}(\H)).
\]
We define $u$ to be its limit.
Since $2m > \frac{\dim(\H)}{2}$, it follows from the Sobolev embedding theorem that $u \in C^0([0, \e) \times \H)$, which implies that
\begin{align*}
	Pu &= \lim_{k \to \infty} P v_{k+1} = \lim_{k \to \infty}f(v_k) = f(u), \\
	u|_{t = 0} &= \lim_{k \to \infty} v_k|_{t = 0} = u_0.
\end{align*}
This proves the assertion.
\end{proof}

\begin{proof}[Finishing the proof of Theorem \ref{mainthm4}]
Let $u$ be the solution given by Lemma \ref{le: convergence_sequence}, associated with the fixed constants $m, N$.
Let us choose different $\tilde m$ and $\tilde N$ such that
\[
	\tilde N > D_{\tilde m}.
\]
By Lemma \ref{le: convergence_sequence} there is a $\tilde T > 0$ and a $\tilde u \in C^0([0, \tilde T], H^{2\tilde m+1}(\H)) \cap C^1([0, \tilde T], H^{2\tilde m}(\H))$ such that
\begin{align*}
	P\tilde u &= f(\tilde u), \\
	\tilde u|_{t = 0} &= u_0.
\end{align*}
We claim that $T = \tilde T$ and $u = \tilde u$.
Let $\hat m := \min(m, \tilde m)$ and $\hat T := \min(T, \tilde T)$. 
It follows that $u - \tilde u \in C^0([0, \hat T], H^{2\hat m}(\H)) \times C^1([0, \hat T], H^{2\hat m-1}(\H))$ and
\begin{align*}
	P(u - \tilde u) &= f(u) - f(\tilde u), \\
	(u - \tilde u)|_{t = 0} &= 0.
\end{align*}
Define $\a \in C^\infty([0, \e) \times \H)$ by
\[
	\a(t, x) := \int_0^1 f'(\tau u(t, x) + (1-\tau)\tilde u(t, x))d\tau.
\]
It follows that 
\begin{align*}
	P(u-\tilde u) &= f(u) - f(\tilde u) \\
		&= \int_0^1 f'(\tau u + (1-\tau) \tilde u)d\tau (u - \tilde u) \\
		&= \a (u- \tilde u).
\end{align*}
We conclude that
\begin{align*}
	P(u - \tilde u) - \a (u- \tilde u) &= 0, \\
	(u - \tilde u)|_{t = 0} &= 0.
\end{align*}
Since $P - \a$ is an admissible wave operator in the sense of Definition \ref{def: admissible}, Theorem \ref{mainthm2} implies that $u - \tilde u = 0$ as claimed.
It follows therefore that 
\[
	u \in C^0([0, \hat T], H^{2\max(m, \tilde m) + 1}(\H)) \cap C^1([0, \hat T], H^{2\max(m, \tilde m)}(\H)).
\]
A standard argument continuation argument for the Cauchy problem for semi-linear wave equations, using the energy estimate Theorem \ref{thm: Energy1} shows now that $\hat T = T = \tilde T$.
Since $\tilde m$ was arbitrary, it follows that
\[
	u \in C^0([0, T], H^{2m}(\H)) \times C^1([0, T], H^{2m-1}(\H))
\]
for all $m \in \N$ such that $2m > \frac{\dim(\H)}{2}$.
If we put $l = 0$ let $k \to \infty$ in Lemma \ref{le: bounded_sequence}, we get
\begin{equation}\label{eq: estimate_at_horizon_nonlinear}
	\norm{(u - w^N)(t, \cdot)}_{2m+1} + \sqrt t \norm{\d_t(u - w^N)(t, \cdot)}_{2m} \leq C_1 t^{N + \frac32}
\end{equation}
for any $m, N \in \N$ such that $2m > \frac{\dim(\H)}{2}$ and $N > D_m$.

Let us again write $P = \psi \d_t^2 + L_1 \d_t + L_2$, where $L_1$ and $L_2$ are differential operators in $\H_t$-direction of order $1$ and $2$ respectively.
Since $\psi(t, \cdot) > 0$ for all $t > 0$ we conclude that
\[
	\d_t^2 u = \frac 1 \psi (f(u) - L_1 \d_t u - L_2 u) \in C^0((0, T], H^{2m}(\H)),
\]
for all $m \in \N$. 
Iterating this for higher derivatives, we conclude that $u \in C^\infty((0, T] \times \H)$.

What remains is the regularity at the Cauchy horizon.
The idea is combine estimate \eqref{eq: estimate_at_horizon_nonlinear} with the equation
\begin{equation} \label{eq: reg_eq}
	\psi \d_t^2(u - w^N) = - L_1 \d_t(u - w^N) - L_2 (u-w^N) + (f(u) - f(w^N)) + (f(w^N) - Pw^N),
\end{equation}
for arbitrary $N \in \N$.
By Lemma \ref{le: non-linearity_estimate}, there are constants $C_2, C_3 > 0$ such that
\begin{align*}
	\norm{(f(u) - f(w^N))(t, \cdot)}_{2m} 
		&\leq C_2 \norm{(u - w^N)(t, \cdot)}_{2m} \\
		&\leq C_3t^{N + \frac32}
\end{align*}
for all $t \in [0, T]$.
From Lemma \ref{le: non-linear_asymptotic_expansion} we know that there is a constant $C_4 > 0$ such that
\[
	\norm{f(w^N) - Pw^N}_{2m} \leq C_4t^{N + 1}.
\]
Inserting these observations in \eqref{eq: reg_eq} and applying estimate \eqref{eq: estimate_at_horizon_nonlinear} proves that 
\[
	\norm{\d_t^2(u-w^N)}_{2m-1} \leq C_5 t^{N}.
\]
Continuing to calculate higher deriviatives $\d_t^j(u - w^N)$ using \eqref{eq: reg_eq}, one shows in a straightforward manner that
\[
	\norm{\d_t^j(u-w^N)}_{2m - (j - 1)} \leq C_{5, j} t^{N +(2 - j)}
\]
for all $j$ such that $j-2 < N$ and $j - 1 \leq 2m$.
Since we know that we can increase $m$ and $N$ whenever necessary, this shows that $u \in C^\infty([0, T] \times \H)$ as claimed.
This concludes the proof.
\end{proof}

\end{sloppypar}
\end{document}